\documentclass[reqno]{amsart}
\usepackage{amssymb}
\usepackage{booktabs}
\usepackage[hidelinks]{hyperref}

\newtheorem{theorem}{Theorem}[section]
\newtheorem{proposition}[theorem]{Proposition}
\newtheorem{lemma}[theorem]{Lemma}
\newtheorem{corollary}[theorem]{Corollary}

\newtheorem{definition}[theorem]{Definition}

\theoremstyle{remark}
\newtheorem{remark}[theorem]{Remark}

\numberwithin{equation}{section}

\begin{document}

\title[Affine Pieri Rule and Fusion Rings]
{Affine Pieri Rule for Periodic Macdonald Spherical Functions and Fusion Rings}

\author{J.F.  van Diejen}

\address{
Instituto de Matem\'aticas, Universidad de Talca,
Casilla 747, Talca, Chile}

\email{diejen@inst-mat.utalca.cl}

\author{E. Emsiz}

\address{Delft Institute of Applied Mathematics,  Delft University of Technology, P.O. Box 5031, 2600 GA Delft, the Netherlands}
\email{e.emsiz@tudelft.nl}

\author{I.N. Zurri\'an}

\address{FAMAF-CIEM, Universidad Nacional de C\'ordoba, C\'ordoba 5000, Argentina}

\email{zurrian@famaf.unc.edu.ar}


\thanks{Work was supported in part by the {\em Fondo Nacional de Desarrollo
Cient\'{\i}fico y Tecnol\'ogico (FONDECYT)} Grants \# 1170179 and \# 1210015}

\date{August, 2021}

\begin{abstract} 

Let $\hat{\mathfrak{g}}$ be an untwisted affine Lie algebra or the twisted counterpart thereof (which excludes the affine Lie algebras of type $\widehat{BC}_n=A^{(2)}_{2n}$).
We present an affine Pieri rule for a basis of periodic Macdonald spherical functions associated
with $\hat{\mathfrak{g}}$.
In type $\hat{A}_{n-1}=A^{(1)}_{n-1}$ the formula in question reproduces an affine Pieri rule for cylindric Hall-Littlewood polynomials due to Korff, which at $t=0$ specializes in turn to a well-known Pieri formula in the fusion ring of genus zero {$\widehat{\mathfrak{sl}}(n)_c$}-Wess-Zumino-Witten conformal field theories.
\end{abstract}

\maketitle


\section{Introduction}\label{sec1}
The  Hall polynomials form a generalization of the Littlewood-Richardson coefficients that provide the structure constants of the classical Hall algebra in the basis of  Hall-Littlewood polynomials; these structure constants (which are polynomial in the Hall-Littlewood parameter) are known to enjoy a very intricate combinatorics  \cite[Chapters II, III]{mac:symmetric}. Indeed, the Hall algebra and its generalizations in terms of quivers turn out to encode
a host of combinatorial, algebra-geometric, and representation-theoretic data \cite{bar-war:hall-littlewood,mac:symmetric,sch:lectures,whe-zin:hall}. Recently, Korff introduced an affine analog of  the
Hall polynomials; these arise as  structure constants of a $t$-deformation of the fusion ring (a.k.a. Verlinde algebra) for {$\widehat{\mathfrak{sl}}(n)_c$}-Wess-Zumino-Witten conformal field theories with respect to a natural basis built from cylindric Hall-Littlewood polynomials \cite{kor:cylindric}.
While to date the precise geometric and/or representation-theoretic interpretation of this  $t$-deformed fusion ring has yet to be disclosed, indications of an intimate relation with the deformed
Verlinde algebras in \cite{tel:k-theory,tel-woo:index} have been noticed \cite{guk-pei:equivariant,kor:cylindric,oku-yos:gauged}.

At $t=0$ the Hall-Littlewood polynomials become Schur polynomials. The corresponding  Littlewood-Richardson coefficients
\cite[Chapter I.9]{mac:symmetric} and their affine counterparts, which arise
as fusion coefficients for {$\widehat{\mathfrak{sl}}(n)_c$}-Wess-Zumino-Witten conformal field theories \cite[Chapter 16]{dif-mat-sen:conformal}, have received massive attention across the mathematics literature because of their rich combinatorics and profound applications in representation theory and Schubert calculus, cf. e.g.  \cite{ful:young} and \cite{goo-wen:littlewood,gep:fusion,kor-stro:slnk,mor-sch:combinatorial} as well as further references therein. Korff's $t$-deformation  is different from the $q$-deformed fusion ring in \cite{fod-lec-oka-thi:ribbon}, which recovers
the  {$\widehat{\mathfrak{sl}}(n)_c$}-Wess-Zumino-Witten fusion ring at the value $q=1$.
 Of special interest is in this connection  the well-known fact that
the closely related  {$\widehat{\mathfrak{gl}}(n)_c$}-fusion ring amounts to a $q=1$ degeneration of the small quantum cohomology ring
of the  Grassmannian of  $n$-dimensional linear subspaces in $\mathbb{C}^{n+c}$
\cite{agn:quantum,ber-cio-ful:quantum}.  The structure coefficients of this small quantum cohomology ring in the basis of Schubert classes,
the genus zero 3-point Gromov-Witten invariants,  can  be computed as quantum counterparts of the Littlewood-Richardson coefficients for Schur polynomials \cite{ber:quantum,gep:fusion,int:fusion,rie:quantum,sie-tan:quantum,tam:gromov,vaf:topological,wit:verlinde}.
Various other combinatorial constructions related to the computation  of genus zero 3-point Gromov-Witten invariants have been considered in the literature, e.g. via the structure constants of algebras of symmetric polynomials
in bases of
cylindric Schur polynomials \cite{pos:affine,mac:cylindric}, in bases of
$k$-Schur polynomials
 \cite{lap-mor:quantum,lam-lap-mor-shi:affine,lam-lap-mor-sch-shi-zab:k-schur},
or in bases of noncommutative Schur polynomials in variables from a plactic algebra
\cite{kor-stro:slnk}, respectively.

If at least one of the two factors in the Littlewood-Richardson product consists of
a Hall-Littlewood polynomial attached to a partition with only a single column, then the explicit form of
the pertinent Hall polynomials is given by the Pieri formula
\cite[Chapter III.3]{mac:symmetric}. The affine analog of this Pieri formula  for cylindric Hall-Littlewood polynomials can be found in
\cite[Corollary 7.4]{kor:cylindric}.
The purpose of the present work is to generalize the affine Pieri formula in question from
 {$\widehat{\mathfrak{sl}}(n)_c$ to the case of an arbitrary
  affine Lie algebra $\hat{\mathfrak{g}}$ \cite{kac:infinite}, excluding those of type    $\widehat{BC}_n=A^{(2)}_{2n}$. In other words,
   $\hat{\mathfrak{g}}$ is assumed to be untwisted  or to be the twisted counterpart of an untwisted affine Lie algebra.

Let us recall at this point that from the perspective of Lie algebras the Hall-Littlewood polynomials in $n$ variables are associated with $\mathfrak{sl}(n)$. The corresponding
generalization of these polynomials
to simple Lie algebras of arbitrary  type is given by the Macdonald spherical functions
 \cite{mac:orthogonal,nel-ram:kostka,par:buildings,sch:galleries}, which were constructed originally
 by Macdonald  as spherical functions on $p$-adic symmetric spaces \cite{mac:spherical}.
With the aid of suitable representations of the affine Hecke algebra, the Pieri formula for the Hall-Littlewood polynomials was generalized to a Pieri formula for Macdonald spherical functions of arbitrary simple Lie type in
\cite{die-ems:unitary}. The key to achieve  an analogous generalization of the affine Pieri formula in 
 \cite{kor:cylindric} is to connect with the work in \cite{die:diagonalization}. To this end, we will
 detail briefly how  affine Pieri formulas arise in the context of  \cite{die:diagonalization},  while also
emphasizing in which sense these differ from the usual Pieri formulas for the Hall-Littlewood polynomials in \cite{mac:symmetric}.

Associated with the standard unit basis $e_1,\ldots ,e_n$ for $\mathbb{Z}^n\subset\mathbb{R}^n\subset \mathbb{C}^n$, let us denote $\bar{e}_j=e_j-\frac{1}{n}(e_1+\cdots+e_n)$ ($j=1,\ldots,n$) and $\omega_r=\bar{e}_1+\cdots +\bar{e}_r$ ($r=1,\ldots ,n-1$). 
For $\lambda\in \Lambda^{(n)}=\{ m_1\omega_1+\cdots +m_{n-1}\omega_{n-1}\mid m_1,\ldots,m_{n-1}\in\mathbb{Z}_{\geq0}\}$
the  $\mathfrak{sl}(n)$ Hall-Littlewood polynomial $R_\lambda (x;t)$ with   variable $x=(x_1,\ldots ,x_n)$  and parameter $t$ is defined by the explicit formula
\begin{align*}\label{HL:a}
R_{\lambda} (x ;t) =  
 \sum_ {\sigma\in S_{n}}   C ( x_{\sigma_1},\ldots ,  x_{\sigma_{n}};t)
 x_{\sigma_1}^{\lambda_1}\cdots  x_{\sigma_{n}}^{ \lambda_{n}}  ,
\end{align*}
where
\begin{equation*}\label{Cp:a}
C (x_1,\ldots ,x_n;t)= \prod_{1\leq j<k \leq n} \frac{1-t x_j^{-1}x_k}{1-x_j^{-1}x_k} ,
\end{equation*}
and the summation is meant over all permutations
$\sigma= { \bigl( \begin{smallmatrix}1& 2& \cdots & n \\
 \sigma_1&\sigma_2&\cdots &\sigma_n
 \end{smallmatrix}\bigr)}$  of the symmetric group $S_{n}$. 
When $\mu=\omega_r$, the corresponding $t$-deformed Littlewood-Richardson coefficients
\begin{equation*}
R_\lambda R_\mu =\sum_{\nu\in\Lambda^{(n)}}  \text{c}^{\nu}_{\lambda ,\mu}(t) R_\nu
\qquad (\lambda,\mu\in  \Lambda^{(n)})
\end{equation*}
are given explicitly by the Pieri rule \cite[Chapter III.3]{mac:symmetric}
\begin{equation}\label{pieri}
R_\lambda R_{\omega_r} =c_{\omega_r}(t)
\sum_{\substack{J \subseteq \{ 1,\ldots ,n\} ,\, |J|=r \\ \lambda+\bar{e}_J\in \Lambda^{(n)}}}    R_{\lambda +\bar{e}_J}    \prod_{\substack{1\leq j < k\leq n \\ j\in J, \, k\not\in J \\ \lambda_j=\lambda_k} }  \frac{1-t^{k-j+1}}{1-t^{k-j}}.
\end{equation}
Here $\bar{e}_J=\sum_{j\in J} \bar{e}_j$, $|J|$ denotes the cardinality of $J$, and 
$c_{\omega_r}(t)= S_r(t)S_{n-r}(t)$ with
$S_n(t) =\prod_{1\leq j<k\leq n} \frac{1-t^{k-j+1}}{1-t^{k-j}}$.

Given a positive integral level $c$,
an affine analog of the Pieri formula \eqref{pieri} valid for $\lambda\in
\Lambda^{(n,c)}= \{ m_1\omega_1+\cdots +m_{n-1}\omega_{n-1} \in\Lambda^{(n)}\mid m_1+\cdots +m_{n-1} \leq c \} $
follows from \cite[Theorem 5.1]{die:diagonalization}:
\begin{align}\label{pieri-affine}
&R_\lambda^{(c)} R_{\omega_r}^{(c)} =\\c_{\omega_r}(t)
&\sum_{\substack{J \subseteq \{ 1,\ldots ,n\} ,\, |J|=r \\ \lambda+\bar{e}_J\in \Lambda^{(n,c)} }}    R_{\lambda +\bar{e}_J}^{(c)}   \prod_{\substack{1\leq j < k\leq n \\ j\in J, \, k\not\in J \\ \lambda_j=\lambda_k} } \frac{1-t^{k-j+1}}{1-t^{k-j}}
 \prod_{\substack{1\leq j < k\leq n \\ j\not\in J, \, k\in J \\ \lambda_j=\lambda_k+c} } 
  \frac{1-t^{n+1-k+j}}{1-t^{n-k+j}} . \nonumber
\end{align}
Here $R_\lambda^{(c)} :X^{(n,c)}\to \mathbb{C}$ refers to the Hall-Littlewood polynomial $R_\lambda$ viewed as a function
on a discrete set  $X^{(n,c)}=X^{(n,c)}(t) \subset \mathbb{T}^n=\{  (x_1,\ldots ,x_n) \in \mathbb{C}^n \mid |x_j|=1, j=1,\ldots n\}$. This set consists of points $x_\mu(t)$, $\mu\in \Lambda^{(n,c)}$ that depend analytically on the Hall-Littlewood parameter $t\in (-1,1)$. Specifically,
for $\mu\in \Lambda^{(n,c)}$ and $t\in (-1,1)$ the point
$x_\mu (t)$ is of the form $(e^{i\xi_1},\ldots ,e^{i\xi_n})$ with the vector of angle coordinates $\xi=\xi_\mu (t)$ being defined as the unique global minimum
of the radially unbounded strictly convex Morse function
$\mathcal{V}^{(n,c)}_{\mu }:\mathbb{R}^n\to \mathbb{R}$
\begin{equation}\label{fp}
\mathcal{V}^{(n,c)}_{ \mu }(\xi ) = \sum_{1\le j < k \le n }    \int_0^{\xi_j-\xi_k} v(\text{x})\text{d}\text{x}
 + \sum_{1\leq j\leq n} \left(
{\textstyle \frac{c}{2} } \xi_j^2-2\pi ( \rho_j+\mu_j)\xi_j 
 \right) ,
\end{equation}
where
$\rho=\omega_1+\cdots+\omega_{n-1}$
and
$
v(\text{x}) = 
\int_0^{\text{x}} \frac{1-t^2}{1-2t\cos (\text{y}) +t^2}\text{d}\text{y}$. This Morse function can be loosely thought of as an analog of the fusion potential,
cf. \cite{gep:fusion,bou-rid:presentations}. It is known  that the Hall-Littlewood polynomials $R^{(c)}_\lambda$, 
$\lambda\in \Lambda^{(n,c)} $ form a linear  basis for the algebra  of functions $f: X^{(n,c)}\to\mathbb{C}$  (cf. \cite[Theorem 5.2]{die:diagonalization}), which gives rise to the following affine analog of the Littlewood-Richardson coefficients for the Hall-Littlewood polynomials:
\begin{equation}\label{H-affine}
R_\lambda^{(c)} R_\mu^{(c)} =\sum_{\nu\in \Lambda^{(n,c)}}  \text{c}^{\nu, (c)}_{\lambda ,\mu}(t) R_\nu^{(c)} \qquad (\lambda,\mu\in  \Lambda^{(n,c)}).
\end{equation}
For $\mu=\omega_r$, the explicit form of $
\text{c}^{\nu, (c)}_{\lambda ,\mu}(t) $ is given by the affine Pieri rule in Eq. \eqref{pieri-affine}. 

The structure constants $ \text{c}^{\nu, (c)}_{\lambda ,\mu}(t) $ \eqref{H-affine} constitute a $t$-deformation of the fusion coefficients
for the genus zero  {$\widehat{\mathfrak{sl}}(n)_c$}-Wess-Zumino-Witten conformal field theories, which are recovered at $t=0$.
Indeed, $s_\lambda (x)=R_\lambda (x;0)$ is given by the $\mathfrak{sl}(n)$ Schur character
and the vector of coordinate angles becomes $\xi_\mu (0) = \frac{2\pi}{n+c} (\rho +\mu)$.  The coordinates of the
points in $X^{(n,c)}=X^{(n,c)}(0)$ are thus given by explicit roots of unity:
$x_\mu(0)=e^{ \frac{2\pi i}{n+c} (\rho +\mu)}=(e^{ \frac{2\pi i}{n+c} (\rho_1 +\mu_1)},\ldots ,e^{ \frac{2\pi i}{n+c} (\rho_n +\mu_n)})$, $\mu\in\Lambda^{(n,c)}$.
The basis functions $s_\lambda^{(c)}:X^{(n,c)}\to\mathbb{C}$, $\lambda\in\Lambda^{(n,c)}$, given by\begin{equation*}
s_\lambda^{(c)} (x_\mu) = s_\lambda (e^{ \frac{2\pi i}{n+c} (\rho +\mu)})\qquad (\lambda,\mu\in\Lambda^{(n,c)}),
\end{equation*}
and the associated structure constants
\begin{equation*}
s_\lambda^{(c)} s_\mu^{(c)} =\sum_{\nu\in \Lambda^{(n,c)}}  \text{c}^{\nu, (c)}_{\lambda ,\mu} s_\nu^{(c)} \qquad (\lambda,\mu\in  \Lambda^{(n,c)})
\end{equation*}
for the algebra of functions $f:X^{(n,c)}\to \mathbb{C}$ in this basis, provide a well-studied combinatorial model for
the genus zero  {$\widehat{\mathfrak{sl}}(n)_c$}-Wess-Zumino-Witten fusion ring
 \cite{dif-mat-sen:conformal,gep:fusion,goo-wen:littlewood,kac:infinite,kir:inner,kor-stro:slnk}.
In particular, the corresponding $t=0$ specialization of the affine Pieri formula \eqref{pieri-affine}:
\begin{equation}\label{pieri-affine-schur}
s_\lambda^{(c)} s_{\omega_r}^{(c)} =
\sum_{\substack{J \subseteq \{ 1,\ldots ,n\} ,\, |J|=r \\ \lambda+\bar{e}_J\in \Lambda^{(n,c)}}}    s_{\lambda +\bar{e}_J}^{(c)}   
\end{equation}
is well-known in this context, cf. e.g. \cite[Equation (3.2)]{and-str:fusion},
 \cite[Equations (16.112), (16.121)]{dif-mat-sen:conformal}, \cite[Equation (3.6)]{gep:fusion}, \cite[Proposition 2.6]{goo-wen:littlewood}, and \cite[Theorem 6.2]{sal:fusion}.

It is important to emphasize at this point that
the deformation of the  genus zero  {$\widehat{\mathfrak{sl}}(n)_c$}-Wess-Zumino-Witten fusion ring stemming from Eq. \eqref{H-affine}
is not constructed in exactly the same manner as in \cite[Section 7]{kor:cylindric}. In a nut-shell: both deformations are related
via level-rank duality \cite{dif-mat-sen:conformal,goo-wen:littlewood,nak-tsu:level-rank}, which has not been established for $t\in (-1,1)\setminus \{ 0\}$ and thus  a priori gives rise to two  dual choices for the Hall-Littlewood deformation of the fusion ring.

In order to generalize the affine Pieri formula \eqref{pieri-affine} from  {$\widehat{\mathfrak{sl}}(n)_c$} to other affine Lie algebras,
we present an affine counterpart of the Pieri formula for Macdonald spherical functions of arbitrary simple Lie type  from  \cite{die-ems:unitary}, which stems from the implementation of periodic boundary conditions.
The underlying representations of the affine Hecke algebra that lead to this affine Pieri formula
are inspired by previous constructions for the graded affine Hecke algebra that were
developed in the context of the study of quantum integrable particle models, cf. \cite{gut-sut:completely,ems-opd-sto:periodic} and references therein.
From this perspective, a partial construction for twisted affine Lie algebras can be found in \cite{die-ems:discrete}; here we
apply these techniques to present a combinatorial model to compute the structure constants of deformed
 genus zero Wess-Zumino-Witten fusion rings for both twisted and untwisted affine Lie algebras (excluding those of type 
 $\widehat{BC}_n=A^{(2)}_{2n}$, cf. \cite{die:deformation}).  In line with was remarked at the end of the first paragraph for  {$\widehat{\mathfrak{sl}}(n)_c$}, we
 expect that these deformed fusion rings are isomorphic to
  deformed
Verlinde algebras from \cite{tel:k-theory,tel-woo:index}; for  {$\widehat{\mathfrak{sl}}(2)_c$} this isomorphism is manifest from the explicit
construction in  \cite[Appendices A and B]{and-guk-pei:verlinde}.

The material is organized as follows.
Section \ref{sec2} presents our deformation of the genus zero Wess-Zumino-Witten fusion ring, which is built
from a basis of periodic Macdonald spherical functions. The main result is an affine Pieri rule that permits to compute the structure constants
for the multiplication in the periodic Macdonald spherical basis  by  basis elements attached to weights that are either minuscule or quasi-minuscule. After setting up some further notational preliminaries concerning the affine Weyl group in Section \ref{sec3}, the pertinent structure constants are exhibited
in Section \ref{sec4}. When the deformation parameter vanishes, one finds a corresponding Pieri formula and structure constants for the genus zero Wess-Zumino-Witten fusion ring itself. The bulk of the paper is devoted to the
proof of our Pieri rule via a suitable  representation of the Hecke algebra of the affine Weyl group. Specifically, the affine Hecke algebra is first employed in Section \ref{sec5} to construct an affine intertwining operator acting in the space of complex functions over the weight lattice. 
Via a  standard construction involving the idempotent associated with the trivial representation of the Hecke algebra of the finite Weyl group, the
periodic Macdonald spherical functions arise in Section \ref{sec6} upon acting with the affine intertwining operator.
In Section  \ref{sec7} it is shown that the periodic Macdonald spherical functions give rise to a basis for a finite-dimensional algebra
of functions supported on critical points of a `fusion potential' of the type in Eq. \eqref{fp}.
We apply the affine intertwining operator so as to derive a family of difference operators diagonalized by the basis of periodic Macdonald spherical functions. The action of these difference operators permits us to compute the corresponding structure constants associated with this basis. 
In Section \ref{sec8} the computation in question is carried out explicitly for the particular case of the Pieri formula, and
Section \ref{sec9} outlines how to recover the  structure constants more generally from the action of the difference operators.

\section{Affine Pieri Rule}\label{sec2}

\subsection{Macdonald spherical functions}
Let $V$ be a real finite-dimensional Euclidean vector space with inner product $\langle \cdot ,\cdot \rangle$ spanned by an irreducible reduced crystallographic root system $R_0$.
We write $Q$, $P$, and $W_0$, for the root lattice, the weight lattice, and the Weyl group associated with $R_0$. The semigroup of the root lattice generated by a (fixed) choice of positive roots $R_0^+$
is denoted by $Q^+$ whereas  $P^+$  stands for the corresponding cone of dominant weights (see e.g. \cite{bou:groupes,hum:reflection} for more details concerning root systems). 

The dual root system $R_0^\vee:=\{\alpha^\vee\mid \alpha\in R_0\}$ and its
positive subsystem $R_0^{\vee ,+}$ are obtained from $R_0$ and $R_0^+$ by applying the involution
\begin{equation}\label{check-inv}
x\mapsto x^\vee:=2x/\langle x ,x \rangle\qquad (x\in V\setminus\{0\} ).
\end{equation}

\begin{definition}
For $\lambda\in P^+$, the Macdonald spherical function $M_\lambda: V \to \mathbb{C}$  is the $W_0$-invariant trigonometric polynomial given explicitly by
\begin{equation}\label{msf-decomposition}
  M_\lambda(\xi )=\sum_{v\in W_0} C(v\xi)e^{i\langle v\xi,\lambda\rangle}
 \end{equation}
with
\begin{equation}\label{mcfun} 
 C({{\xi}}):= 
 \prod_{\alpha\in R_0^+} 
   \frac{ 1-t_{\alpha} e^{-i\langle \xi,\alpha\rangle}}  { 1-e^{-i\langle \xi,\alpha\rangle}}.
\end{equation}
Here 
$t: R_0 \longrightarrow \mathbb C$ is  a root multiplicity function such that  $t_{w\alpha}=t_{\alpha}$ for every $w\in W_0$ and $\alpha\in R_0$.
\end{definition}

 For our purposes the range of the root multiplicity function will be restricted such that $t:R_0\to (-1,1)\setminus \{0\}$.

\subsection{Basis of periodic Macdonald spherical functions}

Let  $\varphi$  and   $\vartheta$  denote the highest root and the  highest short root of $R_0^+$, respectively.
We fix an admissible pair $(R_0,\hat{R_0})$ with $\hat{R_0}$ being equal either to $R_0^\vee$ or to $u_\varphi R_0$, where $u_\varphi=\tfrac{2}{\langle \varphi ,\varphi \rangle}$, and with the positive system $\hat R^+_0$ obtained from $R_0^+$.  In particular,  for simply-laced $R_0$ we have that $\hat R_0=R_0^\vee$.
For $\alpha\in R_0$, let $\hat{\alpha}:=\alpha^\vee$ if $\hat{R_0}=R_0^\vee$ and let $\hat{\alpha}:=u_\varphi \alpha$ if $\hat{R_0}=u_\varphi R_0$.
Then  $\alpha^\vee = m_\alpha \hat\alpha$ with
$m_\alpha=2\langle\alpha ,\hat \alpha\rangle^{-1}$, i.e.
$m_\alpha=1$ if $\hat R_0=R_0^\vee$ and $m_\alpha= \tfrac{\langle\varphi,\varphi\rangle}{\langle\alpha,\alpha\rangle}$ if $\hat R_0=u_\varphi R_0$.
It follows that $\{m_\alpha\}_{\alpha\in R_0}=\{1,m_\vartheta\}$.

We denote the highest short root of $\hat R_0^{\vee,+}$ by $-\alpha_0=\vartheta(\hat R_0^\vee)$, so in particular $\alpha_0=-\vartheta$ if $\hat R_0= R_0^\vee$ and $\alpha_0=-\varphi$ if $\hat R_0= u_\varphi R_0$.
We also write $\hat{Q}$, $\hat{Q}^\vee$,  $\hat{P}$ and $\hat{P}^\vee$ for  the root lattice, the co-root lattice, the weight lattice and the co-weight lattice of $\hat{R_0}$, respectively.
In this setup it will turn out natural to extend the domain of the root multiplicity function in a straightforward manner: $t:R_0\cup R_0^\vee \cup \hat R_0\to  (-1,1)\setminus \{0\}$ such that $t_{\hat\alpha}=t_{\alpha^\vee}=t_\alpha$ for all $\alpha\in R_0$.

Given a fixed positive integer  $c>1$, we consider two affine alcoves in $P$ and $\hat P$:
\begin{align}\label{hpc}
{P_c}:=\{ \lambda \in P\mid 0\leq \langle \lambda,\beta\rangle\leq c,\ \forall\beta\in {\hat R_0^+}\},
\\
\label{hpcc}
\hat P_c:={\{ \mu\in \hat P\mid 0\leq \langle \mu ,\alpha\rangle \leq c,\, \forall \alpha\in R_0^+\}} ,
\end{align}
and an associated set of nodes $\mathcal P_c:= \left\{\xi_\mu \mid \mu\in  \hat P_c\right\}$.
Here  $\xi_\mu:=\xi_\mu(t)$ ($\mu \in {\hat P}$) 
is defined as the unique global minimum of a radially unbounded strictly convex Morse function 
$\mathcal{V}_\mu:V\to \mathbb{R}$ of the form
\begin{equation}\label{mf1}
\mathcal{V}_\mu(\xi)=\frac{c}{2}\langle \xi,\xi \rangle-2\pi \langle \hat\rho+\mu,\xi\rangle +\sum_{\alpha\in R^+_0}\frac{2}{\langle\alpha,\hat\alpha^\vee\rangle}\int_0^{\langle \xi,\alpha\rangle}v_\alpha(\text{x})\text{d} \text{x},
\end{equation}
where $\hat\rho :=\rho(\hat R_0)=\frac{1}{2}\sum_{\alpha\in {\hat R_0^+}} \alpha$ 
and
$
v_\alpha(\text{x}):=(1-t_{\alpha}^2)\int_0^\text{x} \frac{\text{d}\text{y}}{1-2t_{\alpha} \cos(\text{y})+t_{\alpha}^2}. \nonumber
$

Let $\mathcal{C}({\mathcal P_c})$ denote the algebra of functions $f:\mathcal P_c\to \mathbb{C}$.
For any $\lambda\in P_c$, the periodic Macdonald spherical function $M_{\lambda}^{(c)}\in\mathcal{C}(\mathcal P_c)$ 
is given by the restriction of $M_{\lambda}$ to the nodes $\mathcal P_c$.

\begin{theorem}[Basis]\label{R1}
The periodic Macdonald spherical functions $M_{\lambda}^{(c)}$, $\lambda\in {P_c}$, form a basis of $\mathcal{C}({\mathcal P_c})$.
\end{theorem} 

\begin{remark}
In Sect. \ref{sec6}  we will introduce the $W_0$-invariant affine Macdonald spherical functions $\Phi_\xi\in \mathcal C(P)$. It will be seen that, for $\xi\in \mathcal P_c$, the lattice function  $\Phi_\xi$ is periodic with respect to  translations over elements in $c\hat Q^\vee\subset P$  and that $M_{\lambda}^{(c)}(\xi)=\Phi_\xi(\lambda)$ for $\lambda\in P_c$ (see Rem. \ref{rem:tau} for more details).
\end{remark}

\subsection{Structure constants}

Theorem \ref{R1} gives rise to an affine analog of the ($t$-deformed) Littlewood-Richardson coefficients:
\begin{equation}\label{LRcoef}
M_\lambda^{(c)} M_\mu^{(c)} =\sum_{\nu\in P_c} \text{c}^{\nu,(c)}_{\lambda,\mu }(t) M_\nu^{(c)}   \qquad (\lambda,\mu\in  P_c).
\end{equation}
 When $\mu$ is minuscule or quasi-minuscule we have an explicit expression for  the structure constants $\text{c}_{\lambda,\mu}^{\nu ,(c)}(t)$.
Let us recall in this connection that a weight $\mu\in P$ is called {\it minuscule} if  $0\leq \langle \mu, \alpha^\vee\rangle\leq 1$ for all $\alpha\in R_0^+$ and {\it quasi-minuscule} if $0\leq \langle \mu, \alpha^\vee\rangle\leq 2$ for all $\alpha\in R_0^+$ with the upper bound being realized only once (i.e., the quasi-minuscule weight is unique and equal to the highest short root $\vartheta$).

 To formulate this explicit expression for the corresponding  structure constants let us put:

$${\mathrm m}_\lambda(e^{i\xi}):=
\sum_{\nu\in W_0  \lambda} e^{i\langle \nu , \xi \rangle}\quad (\lambda\in { P},\, \xi\in V ),
$$
and
\begin{equation}\label{htau}
\hat h_t:=t_{\alpha_0}\, \hat e_t (-\alpha_0^\vee )
\end{equation}
with
\begin{equation}\label{etau}
\hat e_t(\eta ):= t_{\vartheta}^{\langle \hat \rho^\vee_s,\eta\rangle}t_{\varphi}^{\langle  \rho(\hat R_0^\vee)-\hat \rho^\vee_s,\eta\rangle}
=\prod_{\beta\in {\hat R_0^+}} t_{\beta}^{\langle \eta, \beta^\vee\rangle/2 } \quad (\eta\in {\hat Q})
\end{equation}
and where  $\hat \rho^\vee_s = \frac12\sum_{\alpha\in W_0\vartheta\cap R_0^+}\hat\alpha^\vee\in \hat P^\vee$
 (so $\hat e_t(\eta)$ is a Laurent polynomial in $t_\alpha$, cf. e.g. \cite{mac:poincare}).

\begin{theorem}[Affine Pieri Rule]\label{R2}
For $\lambda\in P_c$, $\xi\in\mathcal P_c$ and $\omega\in P^+$ minuscule or quasi-minuscule, we have that
\begin{equation}\label{Lomega-sym}
{\mathrm m}_\omega(e^{i\xi}) M_\lambda^{(c)}(\xi )  
 =
U_{\lambda,\omega}(t) M_\lambda^{(c)} (\xi )+ \sum_{\substack{\nu\in W_0\omega\\ \lambda+\nu\in {P_c}  }} 
V_{\lambda,\nu}(t) M_{\lambda + \nu}^{(c)} (\xi ) .
\end{equation} 
Here
\begin{align}\label{Vlambdanu}
V_{\lambda,\nu}(t) &:=  \prod_{\substack{\beta\in {\hat R_0^+}\\ \langle \lambda,\beta\rangle=0 \\ \langle \nu,\beta\rangle > 0 }}
\frac{1-t_{\beta} \hat e_t (\beta) }{1-\hat e_t (\beta)}
\prod_{\substack{\beta\in {\hat R_0^+}\\ \langle \lambda,\beta\rangle=c \\ \langle \nu,\beta\rangle < 0 }}
\frac{1-t_{\beta} \hat h_t \hat e_t (-\beta)}{1-\hat h_t \hat e_t (-\beta)} 
\end{align}
and the coefficient $U_{\lambda ,\omega}(t)$ is given by  Eq. \eqref{Ulambdaomega} below (which implies in particular that $U_{\lambda ,\omega}(t)$ vanishes when $\omega$ is minuscule).  
\end{theorem} 

To state the exact expressions for the  coefficient $U_{\lambda ,\omega}(t)$ in \eqref{Lomega-sym}  and for  $\text{c}^{\nu,(c)}_{\lambda,\mu }(t)$ in Eq. \eqref{LRcoef} when  $\mu$  is (quasi)-minuscule, some more notation regarding the underlying affine Weyl group and affine root system is required.
(A more thorough discussion can 
be found e.g. in \cite{bou:groupes,hum:reflection,mac:affine}).

\section{The affine  Weyl group}\label{sec3}


The {\em affine root system} $R$ associated with the admissible pair $(R_0,\hat R_0)$ is
the set of all \emph{affine roots} $\alpha^\vee+ m_\alpha r c=m_\alpha(\hat\alpha+rc) $ $(\alpha\in R_0, r\in \mathbb Z)$. 
An affine root $a=\alpha^\vee +m_\alpha rc\in R$ will be  regarded as an affine linear function $a:V\to \mathbb{R}$ of the form
$a(x)=\langle x,\alpha^\vee\rangle+ rc $ ($x\in V$, $\alpha\in R$, $r\in\mathbb{Z}$), and gives rise to
an affine reflection $s_a:V\to V$ across the hyperplane $V_a:=\{ x\in V\mid a(x)=0\}$ given by
$s_a(x):=x-a(x)\alpha$.
The choice of positive roots $R^+_0$, with a simple basis $\alpha_1,\ldots ,\alpha_n$, determines the set of affine positive roots 
$R^+:=R_0^{\vee,+} \cup \{ \alpha^\vee+m_\alpha r c \mid  \alpha\in R_0, r\in \mathbb N \}$ 
and a corresponding
basis of affine simple roots  $a_0,\ldots ,a_n$ of the form 
$a_0:=\alpha_0^\vee+c$ and $a_j:=\alpha_j^\vee$ for $j=1,\ldots ,n$. Here  $n$ denotes the rank  of $R_0$ ($=\text{dim}\, V$).
Notice that these conventions imply that the affine root system $R$ is of \emph{twisted} type iff $\hat R_0= u_\varphi R_0$ is not simply-laced and of \emph{untwisted} type otherwise.

The {\em affine Weyl group} $W$ is defined as the group generated by the affine reflections $s_a$, $a\in R$ and contains the finite Weyl group $W_0$ as the subgroup fixing the origin.  It is  an infinite Coxeter  group with the simple affine reflections $s_j := s_{a_j}$ ($j=0,1,\dots, n$) as generators and  subject to the relations
\begin{equation}\label{W-rel2}
(s_js_k)^{m_{jk}}=1,  \qquad  j,k \in\{0,\ldots ,n\}
\end{equation}
Here $m_{jk}=1$ if $j=k$ and
$m_{jk} \in \{   2,3,4,6 \}$  if $j\neq k$
(and the provision that for $n=1$ the order
$m_{10}=m_{01}=\infty$).
In particular, any $w\in W$ can be decomposed  as
\begin{equation}\label{red-exp}
w=s_{j_1}\cdots s_{j_\ell},
\end{equation}
with $j_1,\ldots,j_\ell\in \{ 0,\ldots ,n\}$.  
The \emph{length} $\ell (w)$ is defined as the minimum number of reflections $s_j$ ($j=0,1,\dots,n$) involved in any decomposition \eqref{red-exp} of $w$. Any decomposition \eqref{red-exp} with $\ell=\ell(w)$ is called a {\it reduced expression} of $w$.

A fundamental domain for the action of $W$ on $V$ is given by the  \emph{dominant Weyl alcove}
\begin{equation}\label{alcove}
A_c=\{ x\in V\mid 0\leq \langle x,\beta\rangle\leq c,\ \forall\beta\in {\hat R_0^+}\}  .
\end{equation} 
Furthermore, since our positive scale parameter $c$ is integral-valued the weight lattice $P\subset V$ is stable with respect to the action of $W$ and $P_c$ \eqref{hpc} is a fundamental domain for this restriction. 
Given $x\in V$, we will also write $w_x\in W$ for the \emph{unique} shortest affine Weyl group element such that
\begin{equation}
x_+:=w_x x \in A_c.
\end{equation}
For any $\lambda\in P$ let
\begin{equation}
t[\lambda] := \prod_{a\in R[\lambda]} t_{a'}
\end{equation}
where $(\alpha^\vee+m_\alpha r c)' := \alpha^\vee$ denotes the differential and 
 \begin{equation}\label{e:Slambda}
R[\lambda]:= \{a\in R^+ \mid a(\lambda ) < 0\}.
\end{equation}

The action of $w\in W$ on $V$ induces a dual action on the space $\mathcal{C}(V)$ of functions $f:V\to\mathbb{C}$ given by
\begin{equation}\label{W-action}
(wf)(x):=f(w^{-1}x)\qquad (w\in W,\, f\in \mathcal{C}(V),\, x\in V).
\end{equation}
In  Section~\ref{sec4} we will use that $\mathcal C(P)$ is an invariant subspace under this action.

\section{Affine  Littlewood-Richardson coefficients and fusion rules}\label{sec4}
We are now in the position to make the coefficient $U_{\lambda ,\omega}(t)$ in Theorem~\ref{R2} explicit:
\begin{equation} \label{Ulambdaomega}
U_{\lambda ,\omega}(t)  =
\sum_{\substack{\nu\in W_0\omega \\
 (\lambda + \nu)_+=\lambda   }} t[\lambda + \nu]+ (1-t_{\vartheta}^{-1})
\sum_{\substack{\nu\in W_0\omega \\ w_{\lambda + \nu} \lambda=\lambda   }} d_{\lambda ,\nu} ,
\end{equation}
where 
\begin{subequations}
\begin{equation}\label{c-coef}
d_{\lambda,\nu}:=
\begin{cases}
\theta (\lambda  + \nu) e_t (-\nu ) h_t^{\text{sign} (\langle\lambda ,\hat \nu\rangle )} 
    &  \text{if $\nu\in W_0\vartheta$} \\
0 &   \text{otherwise} 
\end{cases}
\end{equation}
(with the convention that $\text{sign}(0):=0$) and
\begin{equation}\label{hvee}
h_t:=t_{\vartheta}\, e_t (-\alpha_0)\quad\text{with}\quad e_t (\nu):=\prod_{\alpha\in R^+_0} t_{\alpha}^{\langle \nu, \alpha^\vee \rangle/2 } .
\end{equation}
 (cf. Eq. \eqref{etau}). 
Here
 $\theta:P\to \mathbb{N}\cup \{ 0\}$ denotes the function 
 \begin{equation}\label{e:theta}
\theta(\lambda ):= \left| {\{a\in R^+ \mid a(\lambda )=-2\}} \right|.
\end{equation} 
\end{subequations}

\begin{remark}
Observe that $d_{\lambda,\nu}$ is also a Laurent polynomial   in  $t_\alpha$, $\alpha\in R_0$.
We will also see in Lemma~\ref{theta:lem} that  $\theta(\lambda + \nu)= 0$ if $\nu$ is in the orbit of a minuscule weight and therefore it is also possible to write 
$d_{\lambda,\nu}=\theta (\lambda  + \nu) e_t (-\nu ) h_t^{\text{sign} (\langle\lambda , \nu^\vee\rangle )}. $
\end{remark}


A function $t:W\to (-1,1)\setminus \{ 0\}$ satisfying 
$t_{w\tilde{w}}= t_w t_{\tilde{w}} $ if $\ell(w\tilde{w} )=\ell(w)+\ell(\tilde{w} ) $
is called a length multiplicative function. 
We compatibilize this function with the root multiplicative function by setting 
$t_j:=t_{s_j}=t_{\alpha_j}$ for $j=0,1,\dots,n$.  For any finite subgroup $G\subset W$ we consider  the  generalized Poincar\'e series  
\begin{equation}\label{Poincare}
G(t)=\sum_{w\in G} t_w
\end{equation}
  of $G$ associated with the length multiplicative function $t$.

\begin{corollary}[Affine Littlewood-Richardson Coefficients]\label{Cmain}
If $\omega$ is a (quasi)-minuscule weight and  $\lambda,\nu \in P_c$, then the affine Littlewood-Richardson coefficients in \eqref{LRcoef} are given by
\begin{equation}
\text{c}^{\nu,(c)}_{\lambda,\omega}(t)
= 
\begin{cases}
W_{0;\omega}(t) \bigl( U_{\lambda,\omega}(t) - U_{0,\omega}(t)  \bigr) &  \text{if $\lambda=\nu$,}\\
 W_{0;\omega}(t) V_{\lambda,\nu - \lambda}(t) & \text{if $\nu- \lambda\in W_0\omega$,}\\
0 & \text{otherwise}.
\end{cases}
\end{equation}
Here $W_{0;\omega}(t)$ refers to generalized Poincar\'e series~\eqref{Poincare} of 
 $W_{0;\omega}=\{w\in W_0 \mid w\omega=\omega  \}$.
\end{corollary}
\begin{proof}
Applying Theorem \ref{R2} with $\lambda=0$ and using that $M_0(\xi) = W_0(t)$ and $V_{0,\theta}(t)=W_{0}(t)/W_{0;\vartheta}(t)$  yields the identity $M^{(c)}_\omega(\xi) = W_{0;\omega}(t) \bigl({\mathrm m}_\omega(e^{i\xi}) - U_{0,\omega}(t) \bigr)$. Combining this with  Theorem  \ref{R2} entails  the desired result.
\end{proof}
When $R_0$ is of type $A_{n-1}$ and $\omega$ is minuscule, Corollary~\ref{Cmain} reproduces the affine Pieri rule in Eq. \eqref{pieri-affine}.

At $t_\alpha=0$ ($\alpha\in R_0$) the Macdonald spherical functions $M_\lambda$ \eqref{msf-decomposition} specialize to the Weyl characters
\begin{equation}\label{Weyl character}
  \chi_\lambda(\xi )= \delta(\xi)^{-1}\sum_{w\in W_0} (-1)^{\ell(w)}e^{i\langle w\xi,\lambda+\rho\rangle}\qquad 
   (\lambda \in P^+) ,
 \end{equation}
where $\rho=\rho(R_0)$ and $\delta(\xi)$ denotes the  Weyl denominator
\begin{equation}\label{mcfun} 
 \delta({{\xi}})
 =\sum_{w\in W_0} (-1)^{\ell(w)}e^{i\langle w\xi,\rho\rangle}
 = \prod_{\alpha\in R_0^+}  (e^{i\langle \xi,\alpha\rangle/2}-e^{-i\langle \xi,\alpha\rangle/2}) .
\end{equation}
The nodes $\mathcal{P}_c$ are in this situation given explicitly by
 \begin{equation}\label{tau=01}
\xi_\mu(0) :=\frac{2\pi }{h+c} (\hat\rho+\mu)  \qquad  (\mu \in \hat P_c),
\end{equation} 
where $h = h(R)= 1-\langle  \rho,\alpha_0^\vee\rangle$ denotes the Coxeter number of the affine root system $R$. 
Indeed, $\lim_{t\to 0} \xi_\mu(t)=\xi_\mu(0)$ by Lemma~\ref{lemma} and Eq. \eqref{critical:eq} below, since
$\hat\rho_v(\xi)=h \xi$ for $t_\alpha=0$ in view of Schur's lemma.

The corresponding parameter degeneration of the structure constants in $\mathcal{C}(\mathcal{P}_c)$,
\begin{equation}
\chi_\lambda (\xi) \chi_\mu (\xi) =\sum_{\nu\in P_c} \text{c}^{\nu,(c)}_{\lambda,\mu }(0) \chi_\nu (\xi)  \qquad (\lambda,\mu\in  P_c,\,\xi\in \mathcal P_c),
\end{equation}
model the fusion rules of the genus-zero Wess-Zumino-Witten conformal field theories associated with the affine Lie algebra
$\hat{\mathfrak{g}}$ of type $R^\vee=\{2a/\langle a',a'\rangle \mid a\in R\}$ \cite{dif-mat-sen:conformal,hon:fusion,kac:infinite} (cf. also Remark \ref{dynkin:rem} below).
 Corollary \ref{Cmain} gives rise to the following Pieri rule in the fusion ring for $\mu=\omega\in P^+$ minuscule or quasi-minuscule:
\begin{equation*}\label{pieri:weylchars}
\boxed{\chi_\omega(\xi) \chi_\lambda(\xi)  
 =
(N_{0,\omega} - N_{\lambda,\omega}) \chi_\lambda(\xi)+ \sum_{\substack{\nu\in W_0\omega\\ \lambda + \nu\in {P_c}  }} \chi_{\lambda + \nu} (\xi ) 
\qquad  (\lambda \in  P_c,\,\xi\in \mathcal P_c),
}
\end{equation*} 
with 
$$N_{\lambda,\omega} =| \{a_j \mid a_j(\lambda)=0 \text{ and  $\alpha_j\in W_0\omega$} \}| $$
 (so $N_{\lambda, \omega}=0$ if $\omega$ is minuscule). When $R_0$ is of type $A_{n-1}$ and $\omega$ is minuscule, this Pieri rule amounts to  Eq. \eqref{pieri-affine-schur}.

To infer the boxed Pieri rule, one  first  observes that  $\lim_{t\to 0} V_{\lambda,\nu - \lambda}(t)=1$,
 $\lim_{t\to 0} W_{0;\omega}(t)=1$, and
$$
\lim_{t\to 0}  (1-t_\vartheta^{-1})d_{\lambda,\nu}
=
\begin{cases}
-\lim_{t\to 0}  t_\vartheta^{-1}  e_t(-\nu) & \text{if  $\nu\in R_0^-\cap W_0\vartheta$  and $\langle \lambda,\hat \nu\rangle=0$} ,\\
-\lim_{t\to 0}    e_t(-\alpha_0-\nu) & \text{if  $\nu\in R_0^+\cap W_0\vartheta$ and $\langle \lambda,\hat \nu\rangle=c$},
\end{cases}
$$
and where  $R_0^-= - R_0^+=R_0 \backslash R_0^+$.  Since 
$\lim_{t\to 0}  t_\vartheta^{-1} e_t(\alpha)=1$ if $\alpha\in R^+_0$ is a simple root and $0$ otherwise,
 this shows that $\lim_{t\to 0} U_{\lambda,\omega}(t)=-N_{\lambda,\omega}$ in view of  Lemma~\ref{theta:lem}  below.
The  asserted Pieri rule in the fusion ring  is now immediate  from Corollary~\ref{Cmain}.

\begin{remark}\label{dynkin:rem}
If, following \cite[Chapter I]{mac:affine}, we denote by $S(R_0;c)=R_0+ c \mathbb Z$ the (untwisted) affine root system associated with $R_0$. Then the following table identifies
the Dynkin type of the affine root system $R$ and of the affine Lie algebra $\hat{\mathfrak{g}}$
in terms of the admissible pair $(R_0,\hat{R}_0)$, via the classification in
 \cite[Chapter I.3]{mac:affine}:
\begin{equation*} 
\begin{tabular}{@{}|c|cc|c|@{}} 
\toprule
$(R_0,\hat{R}_0)$ &$R$ &$\hat{\mathfrak{g}}$  \\
\midrule
$(R_0, R_0^\vee)$& $S(R_0^\vee;c)$ &  $S(R_0^\vee;c)^\vee $  \\ 
$(R_0, u_{\varphi} R_0)$ & $S(R_0;c/u_\varphi)^\vee $ &  $S(R_0;c/u_\varphi)$  \\ 
\bottomrule
\end{tabular} 
\end{equation*}
\end{remark}

\begin{remark} 
If  $\hat{\mathfrak{g}}$ is untwisted (i.e. $(R_0,\hat{R}_0)=(R_0, u_{\varphi} R_0)$), then 
 $\text{c}_{\lambda,\vartheta}^{\nu,(c)}(0)$ is a nonnegative integer. The same is true when
   $\hat{\mathfrak{g}}$ is twisted (i.e. $(R_0,\hat{R}_0)=(R_0,  R_0^\vee)$ with $R_0$ not simply-laced)  provided  $c$ is not an integer multiple of $\tfrac{\langle\varphi,\varphi\rangle}{\langle\vartheta,\vartheta\rangle}\in\{2,3\}$.
If  $\hat{\mathfrak{g}}$ is twisted and $c$ is an integer multiple of $\tfrac{\langle\varphi,\varphi\rangle}{\langle\vartheta,\vartheta\rangle}\in\{2,3\}$, however,  then 
 $\text{c}_{\lambda,\vartheta}^{\lambda,(c)}(0)=N_{0,\vartheta} - N_{\lambda,\vartheta}=-1$
when $\lambda\in P_c$  is chosen such that $a_j(\lambda)=0$ for all $j\in\{0,1,2,\dots, n\}$ with $\alpha_j\in W_0\vartheta$.  This state of affairs is in agreement with prior observations
in \cite{gin:twisted}  regarding the
occurrence of
negative structure constants in the genus-zero Wess-Zumino-Witten fusion ring 
when the underlying affine Lie algebra  $\hat{\mathfrak{g}}$ is twisted.
\end{remark}

\begin{remark}

Recently in  \cite{die:deformation}  a deformation of the Wess-Zumino-Witten fusion ring of type $\widehat{BC}_n=A_{2n}^{(2)}$ was derived, based on a diagonalization of a finite $q$-boson model with diagonal open end boundary conditions obtained in \cite{die-ems:orthogonality}.  The approach in loc. cit. does not use affine Hecke algebras  and hinges instead on Sklyanin's quantum inverse scattering method (a.k.a. the algebraic Bethe Ansatz method) for  the diagonalization of quantum integrable eigenvalue problems with boundary conditions.



\end{remark}

\section{Affine intertwining operator}\label{sec5}
The remainder of the paper is devoted to the proofs of Theorems \ref{R1} and \ref{R2} with the aid of the
 {\em affine Hecke algebra}  ${H}$.  By definition, $H$ is the unital associative algebra over $\mathbb C$ with invertible generators  $T_0, T_1\dots, T_n$  
such that the following relations are satisfied
 \begin{equation}\label{quadratic-rel}
 (T_j-t_j)(T_j+1)=0\qquad (0\leq j\leq n),
 \end{equation}
 \begin{equation}\label{braid-rel}
 \underbrace{T_jT_kT_j\cdots}_{m_{jk}\ {\rm factors}}
 =\underbrace{T_kT_jT_k\cdots}_{m_{jk}\ {\rm factors}}\qquad (0\leq j\neq k\leq n),
 \end{equation}
where the number of factors $m_{jk}$ on both sides of the braid relation \eqref{braid-rel} is the same as the order of the corresponding braid relation \eqref{W-rel2} for $W$  (see e.g. \cite{hum:reflection,mac:affine}).

For a reduced expression $
w=s_{j_1}\cdots s_{j_\ell},$ let 
$T_w:=T_{j_1}\dots\,T_{j_\ell}$
(which does not depend on the choice of the reduced expression by virtue of
the braid relations). It is known that the elements $T_w$, with $w\in W$, form a basis for $H$ over $\mathbb C$.

The subalgebra of $H$ generated by $T_1, \dots T_n$ is referred to as the \emph{finite Hecke algebra} $H_0$ (associated with $W_0$ and $t$).

   To define the affine intertwining operator we need the following integral-reflection representation of $H$.

\begin{proposition}\label{integral-reflection:thm}
The following defines an action of $ H$ on $\mathcal C(P)$:
\begin{equation}\label{Ia}
 T_j  f=(t_j s_j+(t_j-1)  J_j) f\qquad (f\in \mathcal C(P), j=0,\ldots ,n),
\end{equation}
 where   $J_j:\mathcal C(P)\to \mathcal C(P)$ denotes the operator given by
  \begin{equation}\label{Ij}
(J_j f)(\lambda ) := 
 \begin{cases}
-\sum_{k=1}^{a_j(\lambda)}   f(\lambda-k\alpha_j)   &\text{if}\ a_j(\lambda)> 0 ,\\
0&\text{if}\ a_j(\lambda)=0 ,\\
\sum_{k=0}^{-a_j(\lambda)-1}   f(\lambda+k\alpha_j)
&\text{if}\ a_j(\lambda) <0,
\end{cases} 
\end{equation}
\end{proposition}

\begin {proof}

For any $k\in\{0,1,\dots, n\}$ consider 
the (finite dimensional) parabolic subalgebra $H_k$ of $H$ generated by $T_0, T_1,\dots, T_{k-1}, T_{k+1},\dots, T_n $. The idea of the proof is  to show that, for any (fixed) $k$, $T_j\mapsto I_j$, $j\neq k$ extends to a representation of $H_k$ on $\mathcal C(P)$.  
Here $I_j :=t_j s_j+(t_j-1)  J_j$ denote the operator on the right-hand side of Eq.~\eqref{Ia}.
The fact that $H$ is generated by $T_0,\dots, T_n$ subjected to the braid relations and quadratic relations implies then the proposition. 
 For this let us introduce the vertices $v_0=0,v_1,\dots, v_n$ of the alcove $A_c$ \eqref{alcove}, so in particular 
\begin{equation}
a_j(v_k)=\delta_{j,k}, \qquad \text{for all } j\neq k.
\end{equation}
We fix a $k\in\{0,1,\dots, n\}$ and consider the finite subsystem obtained from $R$ by  considering the vertex $v_k$ as the ``new origin'':
 $R_k = \{(a')^\vee  \mid a \in R, a(v_k)=0\}$.  Then $a\mapsto a'$ defines a  root system isomorphism from  the parabolic subsystem  $\{a \in R \mid a(v_k)=0\}$ of $R$ onto $R_k^\vee$. 
The root system $R_k$ is a finite root system of rank $n$ in $V$, although  not necessarily irreducible. A basis of simple roots for $R_k$ is given by $\alpha_j$,  $j\neq k$.  The map $s_j \mapsto s'_j =s_{\alpha_j^\vee}$, $j\neq k$ defines a group isomorphism from 
the parabolic subgroup $W_k=\langle s_j\mid j\neq k \rangle$ of $W$ to the finite Weyl group $W_0(R_k)$ of $R_k$.  This isomorphism induces a natural isomorphism from 
 the  parabolic subalgebra $H_k=\langle T_j \mid j\neq k\rangle$ of  $H$  to the finite Hecke algebra $H_0(W_0(R_k))$ associated with $W_0(R_k)$ and $t_j$, $j\neq k$.

From $R_k \subset R_0$ follows  that $P\subset P(R_k)$,  where $P(R_k)$ denotes  the weight lattice of $R_k$.  We will also need the $(-v_k)$-translation  $P'_k := - v_k + P \subset P(R_k)$ of $P$. 
For any  $j\neq k$ consider the integral-reflection operator  $I'_j: \mathcal C(P(R_k)) \to \mathcal C(P(R_k))$ associated to the finite root system $R_k$, i.e.  
 $I'_j = t_j s'_j+(t_j-1)  J'_j$ where $J'_j$ is given by the same expression as \eqref{Ij} but with $a_j(\lambda)$ replaced mechanically by $\langle \alpha^\vee_j,\lambda \rangle$. 
Observe that $\mathcal C(P'_k)$ is invariant under the operators $J'_j$ and $I'_j$, $j\neq k$. 
The linear isomorphism $\ell_k: \mathcal C(P) \to \mathcal C(P'_k$) 
defined by $(\ell_k f)(y)=f(v_k +y)$ satisfies  $ \ell_k(s_j f) = s'_j \ell_k(f)$,
$\ell_k(J_j f) = J'_j \ell_k(f)$,  and  therefore also
$\ell_k(I_j f) = I'_j \ell_k(f)$, for all $j\neq k$. 

 By  \cite[Lem. 4.2]{die-ems:unitary}, applied to the finite root system $R_k$
 and restricted to the finite Hecke algebra part, it follows that  $T_j \mapsto I'_j$, $j \neq k$, defines a representation of $H_0(W_0(R_k))$ on $\mathcal C(P(R_k))$ (see also Remark \ref{Prop41DE-Unitary} below). 
 Since $\mathcal C(P'_k)$ is an invariant subspace under this representation it follows that  $T_j \mapsto I'_j$, $j \neq k$ extends to a representation of $H_0(W_0(R_k))$ on $\mathcal C(P'_k)$.
Using the above mentioned isomorphism from $H_k$ to $H_0(W_0(R_k))$ we deduce that  $T_j \mapsto I'_j$, $j\neq k$ extends to a representation of $H_k$ on
  $\mathcal C(P'_k)$.   By taking the pullback of the linear isomorphism $\ell_k$ we conclude  that $T_j \mapsto I_j$, $j\neq k$ extends to representation of  $H_k$ 
  on $\mathcal C(P)$. Since $k$ was arbitrary this finishes the proof, as indicated in the beginning. 
\end{proof}

\begin{remark}\label{Prop41DE-Unitary}
In   \cite[Lem. 4.2]{die-ems:unitary} it was assumed that the underlying finite crystallographic root system $R$ was irreducible and of full rank. However, the proposition holds for \emph{all} finite crystallographic root systems of full rank. If $R$  is decomposed into  disjoint, irreducible and orthogonal subsystems $R_1  \cup \dots \cup R_\ell$, then $H_0(W_0(R))\simeq H_0(W_0(R_1)) \otimes \dots \otimes H_0(W_0(R_\ell))$ and the action of the integral-reflection representation also decomposes on 
$\mathcal C(P)\simeq \mathcal C(P(R_1)) \otimes \dots  \otimes \mathcal C(P(R_\ell))
$, which yields immediately the result. 

Alternatively, the proof of   \cite[Lem. 4.2]{die-ems:unitary}  works verbatim if the requirement that $R$ be irreducible is dropped. 
\end{remark}

The affine intertwining operator  $\mathcal{J}:\mathcal{C}(P)\to\mathcal{C}(P)$ is now defined as follows:
\begin{equation}\label{int-op}
(\mathcal Jf)(\lambda):=t[\lambda]^{-1} (T_{w_\lambda} f)(\lambda_+).
\end{equation}

\begin{proposition}\label{intertwining:prp}
The operator $\mathcal{J}$  is invertible.
\end{proposition}

Proposition \ref{intertwining:prp} is a direct consequence of  Lemma \ref{triangular:prp} below. 
 Given $v,w\in W$ we recall that $v\le w$ in the \emph{Bruhat partial order} on $W$ if $v$ may be obtained by deleting simple reflections from the reduced expression of $w$ (see \cite[Sec. 2.3]{mac:affine}).
For $x\in V$, we denote by $[x]$ the finite set $\{ y\in V\mid y_+=x_+ \text{ and } w_y\le w_x\}$ and by $\text{Conv}\, [x]$ the convex hull of $[x]$.
Now we consider the following partial order $\preceq$ on $P$:
\begin{equation}\label{br-order-P}
\forall \mu,\lambda\in P\quad \mu\preceq \lambda\ \text{iff}\ \text{(i)}\ \lambda-\mu\in Q\ \text{and}\
\text{(ii)}\ \text{Conv}\, [ \mu ]\subseteq \text{Conv}\,[ \lambda ].
\end{equation}

\begin{lemma}\label{triangular:prp}
The action of  $\mathcal{J}$ is triangular  with respect to the above partial order:
\begin{equation}\label{triangular}
(\mathcal{J}f)(\lambda )= \sum_{\mu\in P,\, \mu \preceq \lambda} J_{\lambda,\mu} f(\mu),
\qquad (f\in \mathcal C(P),\lambda\in P)
\end{equation}
for some coefficients $J_{\lambda,\mu}\in\mathbb C$ and with $J_{\lambda ,\lambda}=t[\lambda]^{-1}$.
\end{lemma}

\begin{proof}
We observe that $t[\lambda]=t_{w_\lambda}$ and proceed inductively in the length of $w_\lambda$.
For $\ell (w_\lambda )=0$ clearly $(\mathcal{J}f)(\lambda) =f(\lambda)$.
Next, assuming $\ell(w_\lambda)>0$ let $j$ be such that $w_\lambda s_j <w_\lambda$. Hence
\begin{align} \nonumber
(\mathcal{J}f)(\lambda) =&\, t[\lambda]^{-1}(T_{w_{\lambda}}  f)(\lambda_+)=
t_j^{-1}t[s_j\lambda]^{-1}
(T_{w_{s_j\lambda}} T_{j} f)((s_j\lambda )_+) 
\\ 
 \stackrel{\text{IH}}{=}&
\, t_j^{-1}\sum_{ \mu\in P,\,\mu \preceq s_j\lambda} J_{s_j\lambda ,\mu} (T_{j} f) (\mu )
= 
\sum_{\mu\in P,\, \mu \preceq \lambda} J_{\lambda ,\mu} f(\mu ),
\label{leading:coeff}
\end{align}
where the step IH hinges on the induction hypothesis and the last equality is due to the fact 
that the convex hull of
$ \text{Conv}\, [s_j\lambda ]$ and $s_j(\text{Conv}\, [s_j\lambda ] )$ is contained in $\text{Conv}\, [\lambda]  $  since  $ [s_j\lambda]\cup s_j([s_j\lambda])\subseteq [\lambda]$.

The diagonal coefficient of $$\sum_{ \mu\in P,\,\mu \preceq s_j\lambda} J_{s_j\lambda ,\mu} (T_{j}  f) (\mu )$$ corresponds to term with  $\mu=s_j\lambda$ and the coefficient of $f(\lambda )$ in
$(T_j   f)(s_j\lambda)$  is equal to   $1$ when $s_j\lambda \prec \lambda$ (by Eqs. \eqref{Ia}, \eqref{Ij}).  Hence, upon comparing the coefficients of $f(\lambda)$ on both sides of \eqref{leading:coeff} it is seen that $J_{\lambda,\lambda}=t_j^{-1} J_{s_j\lambda ,s_j\lambda}$, which proves the lemma.	
 \end{proof}

To prepare for the next section, we finish with a convenient characterization of the $W$-invariant subspace of $\mathcal{C}(P)$ in terms of $H$ and $\mathcal{J}$.

\begin{lemma}\label{WRinv}
The $W$-invariant subspace
\begin{equation}\label{WRinv1}
\mathcal{C}(P)^{W}:=\{ f\in\mathcal{C}(P)\mid wf=f,\ w\in W\}
\end{equation}
consists of the functions $f:\mathcal{C}(P)\to\mathbb{C}$ that satisfy
$$\mathcal J T_j\mathcal J^{-1}f=t_j   f \qquad (j=0,\dots,n).$$
\end{lemma}
\begin{proof}
For any $f\in \mathcal C(P)$, $j\in \{0,\ldots n\}$ and $\lambda\in P$ we have
\begin{align*}
(\mathcal{J} T_j  \mathcal J^{-1} & f)(\lambda ) {=} t[\lambda]^{-1} ( T_{w_\lambda} T_j\mathcal J^{-1} f)( \lambda_+ )  \\
{=} &
\begin{cases}
t[\lambda]^{-1}
  (T_{w_\lambda s_j}\mathcal J^{-1} f)( \lambda_+ ) &\text{if}\   a_j (\lambda)\geq 0 \\
 t[\lambda]^{-1} 
 \left( t_j (T_{w_\lambda s_j} \mathcal J^{-1} f)( \lambda_+ )
 +(t_j-1)
 ( T_{w_\lambda}\mathcal J^{-1} f)( \lambda_+ )  \right) &\text{if}\   a_j (\lambda)< 0
 \end{cases} \\
 =  &t_j  t[\lambda]^{-1} (T_{w_\lambda}\mathcal J^{-1} f)(\lambda_+ ) \\
& +
t_j^{\chi(a_j(\lambda ))}
\left(
t^{-1}_{w_\lambda s_j} (T_{w_\lambda s_j}\mathcal J^{-1} f)( \lambda_+ ) -
t[\lambda]^{-1} (T_{w_\lambda}\mathcal J^{-1} f)( \lambda_+ )
\right) \\
{=} & t_j f(\lambda) + t_j^{\chi(a_j(\lambda ))} \left( f(s_j\lambda)-f(\lambda) \right).
\end{align*}
Here $\chi$ denotes the characteristic function of $[0,\infty)$ and we also used that 
$$T_wT_j =
\begin{cases}
 T_{ws_j}  & \text{if $\ell(ws_j)=\ell(w)+1$},\\
 t_j T_{ws_j} +(t_j-1)T_w & \text{if $\ell(ws_j)=\ell(w)-1$},
\end{cases} \\$$
 the relation $\ell (w_\lambda s_j)=\ell (w_\lambda )+1$ if $a_j(\lambda)\ge 0$, $\ell (w_\lambda s_j)=\ell (w_\lambda )-1$ if $a_j(\lambda) < 0$, and  the observation that 
\begin{equation*}
t^{-1}_{w_\lambda s_j} (T_{w_\lambda s_j} f)( \lambda_+ )=t[s_j\lambda]^{-1} (T_{w_{ s_j\lambda}} f)( \lambda_+ ) .
\end{equation*}

Hence, $f$ is $W$-invariant if and only if $\mathcal{J} T_j\mathcal J^{-1} f {=} t_j f$.
\end{proof}

\section{Periodic Macdonald spherical functions}\label{sec6}

For a  $\xi\in V$ we define the \emph{affine Macdonald spherical functions}  function in $\mathcal{C}(P)$:

\begin{equation}\label{msf}
\Phi_\xi := \mathcal{J}\phi_\xi\quad\text{with}\quad \phi_\xi := \sum_{v\in W_0} T_v e^{i\xi} ,
\end{equation}
where $e^{i\xi}$ denotes the plane wave function $e^{i\xi}(\lambda ):=e^{i\langle\lambda ,\xi\rangle}$ ($\lambda\in P$).

The plane waves decomposition for $\phi_{{\xi}}$ \eqref{msf} in the next theorem is a known result, see  \cite[Thm. 1]{mac:spherical1} and (with more details) \cite[(4.1.2)]{mac:spherical} or also \cite[Thm. 2.9(a)]{nel-ram:kostka} and \cite[Thm. 6.9]{par:buildings}). To keep our presentation self contained we include a brief verification based on the representation from Proposition~\ref{integral-reflection:thm}.
For any $w\in W$ we define the finite set 
$$R(w):=R^+ \cap w^{-1}(R^-)$$
with $R^-=-R^+=R\backslash R^+$. 
Let us observe that the cardinality of $R(w)$ is equal to the length of $w$ 
and that for any $\lambda\in P$ we have that $R(w_\lambda)=R[\lambda]$ 
(the reader may consult \cite[Section 2.2]{mac:affine}
and \cite[(2.4.4)]{mac:affine}, respectively).
It is also clear that for any $w\in W_0$ one has $R(w)=R_0^+ \cap w^{-1}(R_0^-)$ and where (recall) $R_0^-= -R^+_0$.
\begin{proposition}\label{pw}
The function $\phi_\xi$, for
\begin{equation}\label{Vreg}
\xi\in V_{\text{reg}}:=\{ \xi \in V\mid \langle \xi ,\alpha \rangle \not\in 2\pi\mathbb{Z},\, \forall \alpha\in R_0^+\} ,
\end{equation}
 decomposes as the following linear combination of plane waves
\label{cf-decompositions}
\begin{equation}\label{cf-decomposition}
  \phi_{{\xi}} =\sum_{w\in W_0} C(w{\xi })\mathbf{e}^{iw {\xi}} ,
  \end{equation}
with
\begin{equation}\label{cfun} 
 C({{\xi}}):= 
 \prod_{\alpha\in R_0^+} 
   \frac{ 1-t_{\alpha} e^{-i\langle \xi,\alpha\rangle}}  { 1-e^{-i\langle \xi,\alpha\rangle}}  .
\end{equation}
In particular, $\phi_\xi(\lambda)=M_\lambda(\xi)$ $(\lambda\in P^+, \xi \in V_{reg}).$
\end{proposition}
\begin{proof}

From the action of $T_j$ ($j=1,\dots,n$) \eqref{Ia} we have  that for any $\boldsymbol{\xi}\in V_{\text{reg}}$ 
\begin{subequations}
\begin{equation}\label{I-plane}
\begin{split}
 T_j \mathbf{e}^{i{\xi }} = \text{b}_j(s_j\boldsymbol{\xi}) \mathbf{e}^{i\boldsymbol{\xi}}  +\text{c}_j(s_j\boldsymbol{\xi}) \mathbf{e}^{i s_j \boldsymbol{\xi}} = \text{b}_j(-\boldsymbol{\xi}) \mathbf{e}^{i\boldsymbol{\xi}}  +\text{c}_j(-\boldsymbol{\xi}) \mathbf{e}^{i s_j \boldsymbol{\xi}} ,
\end{split}
\end{equation}
with
\begin{align*}
\text{c}_j(\boldsymbol{\xi}) &=  \frac{ 1-t_{j} e^{-i\langle \xi,\alpha_j\rangle}}  { 1-e^{-i\langle \xi,\alpha_j\rangle}},&
\text{b}_j(\boldsymbol{\xi}) &= t_j-\text{c}_j(\boldsymbol{\xi})= \text{c}_j(-\boldsymbol{\xi})-1.
\end{align*}
Since the stabilizer of ${\xi}\in V_{\text{reg}}$ for the action of $W_0\ltimes2\pi Q^\vee$  is trivial, all the vectors $w{\xi}$, for  $w\in W_0$,  are different to each other modulo $2\pi Q^\vee$. Then, for any ${\xi}\in V_{\text{reg}}$ the plane waves ${e}^{iw\boldsymbol{\xi}}$,  $w\in W_0$, are linearly independent in $\mathcal{C}(P)$. Therefore, the function $\phi_\xi$ may be written as
\begin{equation}\label{pw:exp}
  \phi_{\boldsymbol{\xi}}=\sum_{w\in W_0} C_w(\boldsymbol{\xi})\mathbf{e}^{iw{\xi }},
\end{equation}
for some unique coefficients $C_w(\boldsymbol{\xi})\in\mathbb{C}$.

It follows from Eq. \eqref{I-plane} that for any reduced expression $w=s_{j_\ell}\dots s_{j_1}\in W_0$ the action of $T_w$ on $e^{i\xi}$ is of the form
 \begin{equation}\label{lo}
T_w  \mathbf{e}^{i{\xi }} =\left( \prod_{1\leq k\leq \ell} c_{j_k}(s_{j_k}\cdots s_{j_2}s_{j_1}{\xi}) \right) {e}^{i w{\xi }} \quad +\text{l.o.},
\end{equation}
\end{subequations}
for some coefficients $c_{j_k}$ and  $\text{l.o.}$ is a linear combination of plane waves $\mathbf{e}^{i v{\xi }}$ with $v<w$ in the Bruhat partial order on $W_0$.

Let $w_0$ be the longest element of $W_0$.
Applying the above identity to a reduced expression $w_0=s_{j_\ell}\dots s_{j_1}$ (so $\ell=\# R_0^+$) we conclude that
\begin{align*}
C_{w_0}(\xi)=& \prod_{1\leq k\leq \ell} c_{j_k}(-s_{j_{k-1}}\cdots s_{j_2}s_{j_1}\boldsymbol{\xi}) 
=\prod_{1\leq k\leq \ell} \frac{1-t_{j_k} e^{i\langle \xi,s_{j_1}  s_{j_2} \cdots s_{j_{k-1}} \alpha_{j_k}\rangle}}{1- e^{i\langle \xi,s_{j_1}  s_{j_2} \cdots s_{j_{k-1}} \alpha_{j_k}\rangle}}\\
=&\prod_{\alpha\in R_0^+} \frac{1-t_{\alpha} e^{i\langle \xi, \alpha\rangle}}{1- e^{i\langle \xi, \alpha\rangle}}
=C(-\xi)=C(w_0\xi).
\end{align*}
In the last equality we have used that $R_0^+=R(w_0)$  and in the third that (see e.g. \cite[(2.2.9)]{mac:affine}) $R_0^+=R(w_0)=\{s_{j_1}  s_{j_2} \cdots s_{j_{k-1}} \alpha_{j_k}\mid k=1,2,\dots, \ell \}$.

Let us denote the trivial idempotent 
\begin{equation}\label{e}
{\bf \imath_0}:=\sum_{v\in W_0}  T_v,
\end{equation}
having then $\phi_\xi={\bf \imath_0}\, e^{i\xi}$.
Since $T_j {\bf \imath_0}=t_j {\bf \imath_0}$ we have that $T_j \phi_{{\xi}}=t_j \phi_{{\xi}}$ for $j=1,\ldots,n$. It follows from Eq. \eqref{I-plane} and the linear independence of the plane waves that for ${\xi}\in V_{\text{reg}}$
\begin{equation}\label{rr}
C_{s_jw}(\boldsymbol{\xi}) \text{c}_j(w\boldsymbol{\xi})= C_{w}(\boldsymbol{\xi}) \text{c}_j(-w\boldsymbol{\xi})
 \quad\text{for\ all}\ 
 w\in W_0,\  j\in\{1,\dots, n\} .
\end{equation}

On the other hand, from the product formula in Eq. \eqref{cfun} it follows that for any ${\xi}\in V_{\text{reg}}$ \begin{equation*}
C (s_j\boldsymbol{\xi})\text{c}_j(\boldsymbol{\xi})=
C(\boldsymbol{\xi}) \text{c}_j(-\boldsymbol{\xi})
 \quad\text{for\ all}\ 
  j\in\{1,\dots, n\} .
\end{equation*}
Hence, $C(w{\xi})$ also satisfies  the  recurrence relation in Eq. \eqref{rr}.
Finally, by downward  induction with respect to the Bruhat order starting from the initial condition
$C_{w_0}({\xi}) =C(w_0{\xi})$ (and using that  $c_j({\xi})\neq 0$), we conclude that 
$C_{w}({\xi}) =C(w{\xi})$ for all $w\in W_0$ and any ${\xi}\in V_{\text{reg}}$.
\end{proof}

Before stating the next results, let us recall that  the nodes $\mathcal P_c$  are given by the unique global minima stemming from the strictly convex Morse functions $\mathcal V_\mu$ \eqref{mf1}, $\mu\in \hat P_c$ \eqref{hpcc}. Given  $\mu$, the existence of the global minimum is guaranteed because $\mathcal V_\mu (\boldsymbol{\xi})$    is smooth and $\mathcal V_\mu({\xi})\to +\infty$ for ${\xi}\to\infty$.
Since $\xi_\mu$ is a minimum of $\mathcal{V}_\mu$, it is a solution for $\nabla \mathcal{V}_\mu =0$:
\begin{equation}\label{critical:eq}
c \xi_\mu+\hat\rho_v(\xi_\mu)
 =2\pi( \hat\rho+\mu)
 \qquad \text{where}\quad \hat\rho_v(\xi):= \sum_{\alpha\in R_0^+}v_\alpha(\langle  \xi , \alpha\rangle){\hat\alpha} .
\end{equation}

\begin{lemma}\label{lemma}
The  critical points $\xi_\mu$, ${\mu}\in \hat P_c$ are all distinct and belong to the open alcove  (with respect to the affine action of $W_0\ltimes2\pi Q^\vee$ on $V$) 
 \begin{equation}\label{Avee}
A=\{ \xi\in V \mid 0<\langle \xi ,\alpha\rangle< 2\pi,\,\forall\alpha\in R_0^+\} .
 \end{equation}
Moreover, the position of  $ \xi_\mu$ depends analytically on the parameters $t_{\alpha}\in (-1,1)$.
\end{lemma}
\begin{proof}
\begin{subequations}
From Eq. \eqref{critical:eq} it is clear that one can recover $\mu$ from the value of  $\xi_\mu$, thus $\xi_\mu\neq\xi_\lambda$ if $\mu\neq\lambda$. Also, for any $\beta\in R_0^+$ we have that 
\begin{equation}\label{critical-beta:eq}
c \langle \xi_\mu,\beta\rangle+\langle \hat\rho_v(\xi_\mu),\beta\rangle
 =2\pi\langle  \hat\rho+\mu ,\beta\rangle.
\end{equation}
Since $v_\alpha (\text{x})$ is an odd function it follows moreover that 
\begin{align}
\langle \hat\rho_v(\xi_\mu),\beta\rangle 
=
\frac{1}{2}\sum_{\substack{\alpha\in R_0 \\ \langle \alpha^\vee,\beta \rangle >0}}
\bigl( v_\alpha(\langle  \xi_\mu, \alpha\rangle)-v_\alpha (\langle  s_{\beta^\vee} \xi_\mu, \alpha\rangle) \bigr) \langle{\hat\alpha},\beta \rangle. \label{rho-vb}
\end{align}
From Eqs. \eqref{critical-beta:eq}, \eqref{rho-vb} one deduces that $\langle \xi_\mu,\beta\rangle >0$ for $\mu\in {\hat P}_{c}$. Here one exploits that 
$v_\alpha (\text{x})$ is strictly monotonously increasing and that
\begin{equation}\label{pos}
\langle \xi_\mu,\alpha\rangle-\langle s_{\beta^\vee} \xi_\mu,\alpha\rangle
=
\langle\xi_\mu,\beta\rangle\langle\alpha,\beta^\vee\rangle .
\end{equation}

Moreover,  from Eqs. \eqref{rho-vb}, \eqref{pos} with $\beta=\varphi$ and the quasi-periodicity of the function $v_\alpha (\text{x})$ one deduces that for $\langle \xi_\mu ,\varphi\rangle \ge 2\pi$ we would have
\begin{equation}\label{lhs-bound}
c\langle \xi_\mu,\varphi\rangle + \langle \hat\rho_v (\xi_\mu),\varphi\rangle\geq
2\pi c +\pi
\sum_{\alpha\in R_0^+} \langle \alpha,\varphi^\vee\rangle \langle \varphi,\frac{\alpha^\vee}{m_\alpha}\rangle
 {=} 2\pi (1+c+ \langle \hat\rho,\varphi\rangle ),
\end{equation}
where in the last connection we used that  for {\em any} root multiplicity function $t:R_0\to\mathbb C$ and root $\beta\in R_0$ we have
\begin{equation}\label{schur:eq}
\sum_{\alpha\in R_0^+} t_{\alpha} \langle \beta,{\alpha^\vee}\rangle \langle\alpha ,\beta^\vee\rangle =\frac{2}{n}\sum_{\alpha\in R_0} t_{\alpha}
\end{equation}
(which follows from the Schur's lemma, cf. the proof of \cite[\text{Lem}.~10.1]{ems-opd-sto:periodic} and \cite[\text{Rem}.~7.4]{die-ems:discrete}) and that 
$$\langle \hat\rho,\varphi\rangle +1=\frac1n\sum_{\alpha\in R_0}\frac{1}{m_\alpha}.$$
Now by combining the Eq. \eqref{lhs-bound}  with Eq. \eqref{critical-beta:eq} for $\beta=\varphi$, we would have
$$c+\langle \hat\rho,\varphi\rangle +1\leq \langle \hat\rho +\mu ,\varphi\rangle
 =\langle \mu ,\varphi\rangle + \langle \hat\rho,\varphi\rangle,$$
which contradicts our assumption that $\mu\in \hat P_{c}$. Hence, one must have that $\langle \xi_\mu ,\varphi\rangle<2\pi$, i.e. $\xi_\mu\in A$.

Finally, it is clear that the critical equation \eqref{critical:eq} is analytic in the parameters $t_{\alpha}\in (-1,1)$. Since the Jacobian of the critical equation is invertible,  the implicit function theorem now ensures that the dependence of the critical point $\xi_\mu$ is also analytic in these parameters.
\end{subequations}
 \end{proof}

\begin{proposition}[Periodic Macdonald Spherical Function] \label{inv}
For every ${\mu}\in \hat P_c$  the function  
$\Phi_{{\xi}_\mu}$  belongs to the $W$-invariant subspace $\mathcal{C}(P)^{W}$.
\end{proposition}
\begin{proof}
Since $T_j {\bf \imath_0}=t_j{\bf \imath_0}$ for $j=1,\ldots ,n$ (see Eq. \eqref{e}), we have that 
\begin{equation}\label{W0-invariance}
\mathcal J T_j \mathcal J^{-1} \Phi_\xi=  \mathcal J T_j\phi_\xi=t_j \mathcal J \phi_\xi=t_j  \Phi_\xi .
\end{equation}
Hence, by Lemma \ref{WRinv}, we only need to prove that $\mathcal J T_0\mathcal J^{-1}\Phi_\xi=t_0\Phi_\xi$, or equivalently $T_0\phi_\xi=t_0\phi_\xi$.
For $\xi\in V_{\text{reg}}$, the decomposition in Eq. \eqref{cf-decomposition} together with the explicit action of $T_0$ on $e^{i\xi}$ (cf. Eqs. \eqref{Ia}--\eqref{Ij}) gives us that
\begin{align*}
T_0 \phi_\xi =
 \sum_{v\in W_0} \frac{t_0-1}{1-e^{ i\langle v\xi,{\alpha_0}\rangle}} C(v\xi)e^{iv\xi}   
   + \sum_{v\in W_0} \frac{1-t_0 e^{i\langle v\xi,-{\alpha_0}\rangle} }{1-e^{i\langle v\xi,-{\alpha_0}\rangle}} C(s'_0v\xi) e^{ ic\langle  v\xi,{\alpha_0}\rangle}  e^{iv\xi}.
\end{align*}
Now, by comparing  with the corresponding decomposition of $t_0\phi_\xi$, we have that $T_0\phi_\xi=t_0\phi_\xi$ if for $\xi\in V_{reg}$
\begin{equation*}
e^{i c\langle v\xi,-{\alpha_0}\rangle} =
\frac{C(s'_{0}v\xi)}{C(v\xi)}
\frac{1-t_0e^{i\langle v\xi,-{\alpha_0}\rangle }}{t_0-e^{i\langle v\xi,-{\alpha_0}\rangle } },
\qquad \forall v\in W_0.
\end{equation*}
By substituting the product expansion for $C(\cdot) $ over $R_0^+$ (cf. Eq. \eqref{cfun}) we have
\begin{align*}
\frac{C(s'_0\xi)}{C(\xi)}&= 
\prod_{\substack{\alpha\in R^+_0\\ \langle  -{\alpha_0},\alpha^\vee\rangle>0  }}
\frac{1-t_{\alpha} e^{i\langle \xi,\alpha \rangle }}{t_{\alpha}-e^{i\langle \xi,\alpha \rangle } } \\
&=\frac{1-t_0 e^{i\langle \xi,-{\alpha_0}\rangle }}{t_0-e^{i\langle \xi,-{\alpha_0}\rangle } }
\prod_{\alpha\in R^+_0\setminus\{ -\alpha_0\}}
\left(\frac{1-t_{\alpha} e^{i\langle \xi,\alpha \rangle }}{t_{\alpha}-e^{i\langle \xi,\alpha \rangle } }
\right)^{\langle -{\alpha_0},\hat\alpha \rangle},
\end{align*}
where we used that $\langle -{\alpha_0},\hat\alpha \rangle \in\{0,1\}$ for all $\alpha\in R_0^+\setminus \{-\alpha_0\}$. The relation now may be written as
\begin{equation}\label{BAE}
 e^{ic\langle \xi,-v\alpha_0\rangle}
=\prod_{\alpha\in R^+_0}  \left( \frac{1-t_{\alpha}e^{i\langle \xi,\alpha \rangle }}{e^{i\langle \xi,\alpha \rangle} -t_{\alpha} }
\right)^{{\langle -v\alpha_0, \hat\alpha\rangle} },\qquad \forall v\in {W_0} ,
\end{equation}
by using that an overall flip of the signs in the factors at the right-hand side   cancels out because $\prod_{\alpha\in \hat R^+_0}  (-1)^{{\langle \beta^\vee, \alpha\rangle} }=(-1)^{\langle\beta^\vee,2 \hat \rho\rangle}=1$
for all $\beta\in {\hat R_0}$.

To finish this proof, let us observe that if we multiply Eq. \eqref{critical-beta:eq} by the imaginary unit and exponentiate both sides, using that 
$$
v_\alpha(\text{x})=2\arctan\Bigl(\frac{1+t_{\alpha}}{1-t_{\alpha}}\tan \left(\frac{\text{x}}{2}\right)\Bigr)=
i\log\Bigl(\frac{1-t_{\alpha} e^{i\text{x}}}{e^{i\text{x}}-t_{\alpha}}\Bigr),
$$ then it follows that $\xi_\mu$ is indeed a solution for Eq. \eqref{BAE}.
\end{proof}

\begin{remark}
In \cite[Eqn. (5.12a)]{die-ems:unitary} a Macdonald spherical function $\Phi_\xi \in \mathcal C(P)^{W_0}$ was introduced in terms of an intertwiner operator built  up (essentially) from an integral-reflection representation of the finite Hecke algebra $H_0$.  In contrast, the affine Macdonald spherical function  $\Phi_\xi \in \mathcal C(P)^{W_0}$~\eqref{msf} is based on the  affine intertwiner operator $\mathcal J$~\eqref{int-op}, built up  from the integral-reflection representation of the affine Hecke algebra $H$.
\end{remark}

\begin{remark}\label{rem:tau}

For $y\in V$, let us denote by $\tau_y:V\to V$ the translation determined by the action $\tau_y(x):=x + y$. Then the affine Weyl group admits the alternative presentation $W=W_0\ltimes \tau(c \hat Q^\vee)$  because  $ s_{\alpha^\vee}s_{\alpha^\vee+m_\alpha rc}=\tau_{c r  \hat\alpha^\vee}$ for $\alpha\in R_0$, $r\in \mathbb Z$.
 Because of the above proposition it follows that  $\Phi_{{\xi}_\mu}$ ($\mu \in \hat P_c$) is $W_0$-invariant and $c\hat Q^\vee$-periodic, explaining the name \emph{periodic} Macdonald spherical function for  $\Phi_{{\xi}_\mu}$. 
\end{remark}

\begin{remark}
From the proof of Proposition \ref{inv}, it is clear that for every $\mu\in \hat P$ the vector $\xi=\xi_\mu$ solves the following algebraic system of equations of Bethe type
\begin{equation*}
 e^{ic\langle \xi,\beta^\vee\rangle}
=\prod_{\alpha\in R^+_0}  \Bigl( \frac{1-t_{\alpha}e^{i\langle \xi,\alpha \rangle }}{e^{i\langle \xi,\alpha \rangle} -t_{\alpha}}
\Bigr)^{{\langle \hat\alpha,\beta^\vee\rangle} },\qquad \forall \beta\in {\hat R_0}.
\end{equation*}
Indeed, at $\xi=\xi_\mu$  Eq. \eqref{BAE} is satisfied and the short roots of $\hat R_0^\vee$ generate ${\hat R^\vee_0}$ over~$\mathbb{Z}$. 
\end{remark}

\section{Proof of Theorem \ref{R1} (Basis)}\label{sec7}
For any $\omega\in P^+$ we consider the \emph{free} operator $L_{\omega;1}:\mathcal{C}(P)\to\mathcal{C}(P)$ given by
\begin{equation}\label{free-laplacian}
(\mathrm L_{\omega;1}f)(\lambda):=\sum_{\nu\in W_0\omega} f(\lambda+\nu)
\end{equation}
 and the operator $L_\omega:\mathcal{C}(P)\to\mathcal{C}(P)$ given by
\begin{equation}\label{Lp}
L_\omega = \mathcal{J}  L_{\omega;1}  \mathcal{J}^{-1}  .
\end{equation}

For the isotropy group $W_{\lambda}$ of $\lambda\in {P_c}$ in $W$, 
Macdonald's product formula for the generalized Poincar\'e series of the Coxeter group associated with the length multiplicative function $t$ \cite{mac:poincare} tells us that  
\begin{align}\label{poincare-stab}
W_{\lambda}(t)& =  \sum_{w\in W_{\lambda}}t_w
 =  \prod_{\substack{\alpha\in {\hat R_0^+}\\ \langle \lambda,\alpha\rangle=0}}
\frac{1-t_{\alpha} \hat e_t (\alpha) }{1-\hat e_t(\alpha)}
\prod_{\substack{\alpha\in {\hat R_0^+}\\ \langle \lambda,\alpha\rangle=c  }}
\frac{1-t_{\alpha} \hat h_t \hat e_t (-\alpha)}{1-\hat h_t \hat e_t(-\alpha)},
\end{align} 
with $\hat h_t$ and $\hat e_t$ given by \eqref{htau}, \eqref{etau}.
Armed with this identity, we can readily infer that for any 
$\mu\in \hat P_c$ the function  $\Phi_{\xi_\mu}$ is nonzero in 
$\mathcal{C}({P_c})\cong\mathcal{C}(P)^{W}$. Indeed, from Proposition \ref{pw}, Lemma \ref{lemma},
Proposition \ref{inv} and the trivial action of $\mathcal{J}$ on
$\mathcal{C}({P_c})$, it follows that
\begin{equation}\label{Msf-evaluation}
\Phi_{\xi} (\lambda)=M_\lambda (\xi )=M^{(c)}_\lambda(\xi )\quad\text{for any}\ \xi\in\mathcal{P}_c \ \text{and}\ \lambda\in P_c.
\end{equation}
In particular, at $\lambda=0$ this yields that
\begin{align*}
\Phi_{\xi} (0)& = \sum_{v\in W_0} C(v\xi )=\sum_{v\in W_0} 
\prod_{\alpha\in R^+_0} 
 \frac{1-t_{\alpha} e^{-i\langle v\xi ,\alpha\rangle }}{1-e^{-i\langle v\xi ,\alpha\rangle}} \\
&\stackrel{\star}{=}
\sum_{v\in W_{0}}t_v 
=\prod_{\alpha\in R_0^+}
\frac{1-t_{\alpha}  e_t (\alpha) }{1-e_t (\alpha)}> 0 ,
\end{align*}
where we used Macdonald's identity from Ref. \cite[\text{Thm.} (2.8)]{mac:poincare} for the $\star$ equality.

\begin{proposition}[Completeness of the Periodic Macdonald Spherical Functions] \label{diagonal:thm}
The restriction of the functions $\Phi_{\xi_\mu}$, $\mu\in {\hat P}_{c}$, constitutes a basis for $\mathcal{C}({P_c})$ that diagonalizes the commuting operators $L_\omega$ simultaneously:
\begin{equation}\label{eigen}
L_\omega \Phi_{\xi_\mu} = \mathrm m_\omega(e^{i\xi_\mu }) \Phi_{\xi_\mu} \qquad (\omega \in P^+, \mu\in {\hat P}_{c}) .
\end{equation}

\end{proposition} 
\begin{proof}
For any $\omega\in P^+$ the action of $L_{\omega;1}$ on a plane wave yields
\begin{equation*}
L_{\omega;1}  e^{iv\xi}=
{\mathrm m_\omega} (e^{i\xi}) e^{iv\xi}\qquad (v\in W_0).
\end{equation*}
Hence, given  $\mu \in \hat P_c$ it follows that at $\xi=\xi_\mu$:
\begin{eqnarray*}
L_\omega \Phi_\xi = \mathcal{J} L_{\omega;1} \mathcal{J}^{-1}\mathcal{J}\phi_\xi =\mathcal{J} L_{\omega;1}\phi_\xi =  {\mathrm m_\omega} (e^{i\xi}) \Phi_\xi .
\end{eqnarray*}
The upshot is that the nontrivial eigensolutions $\Phi_{\xi_\mu}$, $\mu\in \hat P_c$ in Eq. \eqref{eigen}
must be linearly independent in $\mathcal{C}(P_c)$, in view of Lemma \ref{lemma} and the well-known fact that
 the $W_0$-invariant trigonometric polynomials $m_\omega(e^{i\xi})$, $\omega\in P^+$,  separate the points of the fundamental alcove $A$.

To finish the proof, it suffices to verify that $\dim \mathcal{C}({P_c})=|P_c|=|{\hat P}_c|$, for this confirms that the eigenfunctions $\Phi_{\xi_\mu}$ ($\mu\in {\hat P}_{c}$) form a basis of $\mathcal{C}({P_c})$.
To this end we first observe that  $P_c$ consists of all nonnegative integral combinations  $c_1\omega_1+\dots+c_n\omega_n$ of the fundamental weights of $R_0$  satisfying  $c_1m_1+\dots+c_nm_n\leq c$, where the positive integers $m_1,\dots, m_n$ refer to the coefficients of the highest root  $-\alpha_0^\vee$ of $\hat R_0$ in the simple basis
 $\alpha^\vee_1,\dots, \alpha^\vee_n$ of $R_0^\vee$. Similarly, $\hat P_c$ consists of all nonnegative integral combinations  $c_1\hat\omega_1+\dots+c_n \hat \omega_n$ of the fundamental weights of $\hat R_0$  satisfying  $c_1\hat m_1+\dots+c_n \hat m_n\leq c$ and where the positive integers $\hat m_1,\dots, \hat m_n$ now refer to the coefficients of $\varphi$ in the simple basis
 $\hat \alpha_1^\vee,\dots, \hat\alpha_n^\vee$ of $\hat R_0^\vee$. If  $\hat R_0=u_\varphi R_0$ then clearly $\hat m_j = m_j$ for all $j$  and if $\hat R_0=R_0^\vee$ then the $\hat m_1,\dots, \hat m_n$ are a permutation of the $m_1,\dots, m_n$. So in both cases $|P_c|=|{\hat P}_c|$, which completes the proof of the proposition. 
\end{proof}

Proposition \ref{diagonal:thm} guarantees that the square matrix
$\left[M_\lambda^{(c)}(\xi_\mu)\right]_{\lambda\in P_c,\mu\in\hat{P}_c}$ is of full rank, which finishes the
proof of Theorem \ref{R1}.

\section{Proof of  Theorem \ref{R2} (Affine Pieri Rule)}\label{sec8}

For any function $f\in\mathcal{C}(P)^{W}$,  $\omega\in P^+$ (quasi)-minuscule and $\lambda\in P_c$ we have that
\begin{align*}
  (L_\omega f)(\lambda)  {=}& (\mathcal{J}   L_{\omega;1}  \mathcal{J}^{-1}f) (\lambda)
  {=}
(  L_{\omega;1}  \mathcal{J}^{-1}f)(\lambda)  \\
{=}&  \sum_{\nu\in W_0\omega} (  \mathcal{J}^{-1}f)(\lambda  + \nu)\\
{=}&
\sum_{\nu\in W_0\omega} t[\lambda + \nu] f(\lambda + \nu)+d_{\lambda,\nu}(1-t^{-1}_{\vartheta}) f(\lambda) .
\end{align*}
The last equality hinges on the following lemma (whose proof is delayed until subsection \ref{subsectionLemmaB}):

 \begin{lemma}\label{Iqm-action:lem}  For any $f\in\mathcal{C}(P)$, $\lambda\in {P_c}$ and
 $\nu\in P_\vartheta^\star:=\{w\eta \mid w\in W_0 \text{ and } \eta\in P \text{ is a minuscule or quasi-minuscule weight} \}$, one has that
  \begin{equation*}\label{Iqm-action:eqn}
( \mathcal J^{-1} f)(\lambda  + \nu ) =   t[\lambda + \nu]  f(\lambda + \nu)  
 + d_{\lambda ,\nu}(1-t^{-1}_{\vartheta})  f(\lambda) ,
 \end{equation*}
where $d_{\lambda ,\nu}$ is taken from  \eqref{c-coef}. 
 \end{lemma}

Since $f$ is $W$-invariant, it follows that
\begin{equation}\label{Lomegatemp}
\begin{split}
(L_\omega f)(\lambda)=&
\left(\sum_{\substack{\nu\in W_0\omega\\ (\lambda+\nu)_+=\lambda }} \,  t[\lambda+\nu]  +
(1-t_{\vartheta}^{-1})\sum_{\nu\in W_0\omega  }  d_{\lambda,	\nu} \right)
\, f (\lambda)
\\
&+
 \sum_{\substack{\nu\in W_0\omega\\ \lambda+\nu\in {P_c}  }} 
\,{\sum_{\substack{\eta\in W_0\omega\\ (\lambda+\eta)_+=\lambda+\nu }}  t[\lambda+\eta]  }\,f (\lambda+\nu).
\end{split}
\end{equation}
The action of $L_\omega $ on $f$ is therefore of the form
$$
(L_\omega f)(\lambda)=
U_{\lambda,\omega}(t) f (\lambda)+ \sum_{\substack{\nu\in W_0\omega\\ \lambda+\nu\in {P_c}  }} V_{\lambda,\nu}(t) f (\lambda+\nu).
$$

The computation of the coefficients $U_{\lambda,\omega}(t)$ and $V_{\lambda,\nu}(t)$ hinges on
 the following lemma  (whose proof is relegated in turn to subsection
\ref{subsectionLemmaA}):
 \begin{lemma}\label{theta:lem}
  For $\lambda\in {P_c}$ and $\nu\in P_\vartheta^\star$, we are in either one of the following two situations:
    $i)$ When $(\lambda + \nu)_+=\lambda$, then $w_{\lambda + \nu}'\nu=\alpha_j$ for some $j\in\{0,\dots, n\}$ with $t_j=t_0$ and $\theta(\lambda + \nu)=0$.
  
 $ii)$ When $(\lambda + \nu)_+\neq \lambda$, then $w_{\lambda + \nu}\in W_{\lambda}$ and
$ \theta(\lambda + \nu)=1$ if  $\nu\in R^-_0\cap W_0\vartheta$ and $\langle \lambda, \hat \nu \rangle =0$,  or if $\nu\in R^+_0\cap W_0\vartheta $ 
and $\langle \lambda, \hat \nu \rangle =c$, while $ \theta(\lambda+\nu)=0$ otherwise.

 \end{lemma}

Indeed, the asserted expression  for
$U_{\lambda,\omega}$ in  \eqref{Ulambdaomega} is  immediate from
Eq.  \eqref{Lomegatemp} and the lemma, while
 the coefficient $V_{\lambda,\nu}(t)$ of $f(\lambda + \nu)$ in Eq. \eqref{Vlambdanu} is  retrieved after a
 short computation:
\begin{equation*}
\sum_{\substack{\eta\in W_0\omega\\ (\lambda + \eta)_+=\lambda + \nu }}
t[\lambda + \eta]
=\sum_{\mu\in W_{\lambda}(\lambda + \nu)}
t[\mu]
=
{W_{\lambda}(t)}/{(W_{\lambda}\cap W_{\lambda+\nu})(t)}
=
V_{\lambda,\nu}(t),
\end{equation*}
where in the last step Macdonald's product formula \eqref{poincare-stab} was used.

Upon combining with Eq. \eqref{eigen}  and recalling   that  for $\lambda\in P_c$ and
$\mu\in \hat P_c$: $\Phi_{\xi_\mu}(\lambda)=M_\lambda(\xi_\mu)=M^c_\lambda(\xi_\mu)$ (cf. Eq. \eqref{Msf-evaluation}), the Pieri formula in Eq. \eqref{Lomega-sym} readily follows.

\subsection{Proof of Lemma \ref{theta:lem}}\label{subsectionLemmaA}
Let $\mu\in P\setminus {P_c}$ and $j\in \{ 0,\ldots ,n\}$ such that $a_j\in R[\mu]$ (recall \eqref{e:Slambda}). Then
 $w_\mu =w_{s_j\mu}s_j$ with $\ell( w_\mu )= \ell (w_{s_j\mu})+1$, and thus
 $R[\mu]=s_j R[s_j\mu]\cup \{a_j\}$ (cf. \cite[(2.2.4)]{mac:affine}, although it is only stated for non-twisted types, it is actually true also for twisted types).
 From \eqref{e:theta} it follows that
 \begin{equation}  \label{theta-rec}
   \theta(\mu)=
   \begin{cases}
   \theta(s_j\mu)+1& \text{if $a_j(\mu)=-2$},\\
   \theta(s_j\mu) & \text{if $a_j(\mu)\neq-2$}.
   \end{cases}
 \end{equation}
Let us consider the situation $\lambda\in {P_c}$, $\nu\in P_\vartheta^\star$ and
 $\lambda + \nu\not\in {P_c}$. If for any $j\in \{ 0,\ldots ,n\}$ such that $a_j\in R[{\lambda+\nu}]$ we define 
 $\tilde{\nu}:=s_j(\lambda + \nu)- \lambda$ (having  $s_j (\lambda +\nu)=\lambda+\tilde{\nu}$), then we are in one of the following cases and thanks to  \eqref{theta-rec} we have the corresponding expressions for $\theta(\lambda+\nu)$ 
 \begin{itemize}
 \item[$(A)$]$a_j(\lambda)=0$ and $ \langle \nu, \alpha^\vee_j\rangle =-1$ (so $a_j(\lambda+\nu)=-1$). Then
 $s_j\in W_{\lambda}$, so
  $\tilde{\nu}=s'_j\nu$ and  $\theta(\lambda+\nu){=}\theta(\lambda+s'_j\nu)$.
 \item[$(B)$] $a_j(\lambda)=0$ and $\langle \nu,\alpha^\vee_j\rangle =-2$ (so $a_j(\lambda + \nu)=-2$). Then
 $s_j\in W_{\lambda}$ and $\nu=-\alpha_j$,
 so $\tilde{\nu}=s'_j\nu=\alpha_j$ and $\theta(\lambda+\nu){=}\theta(\lambda+s'_j\nu)+1$.
 \item[$(C)$] $a_j(\lambda)=1$ and $\langle\nu,\alpha_j^\vee\rangle=-2$  (so $a_j(\lambda + \nu)=-1$).  Then
 $\nu=-\alpha_j$ and $\tilde{\nu}=0$, so
 $w_{\lambda + \nu}=s_j$ and $\theta(\lambda + \nu){=}\theta (\lambda)=0$.
 \end{itemize}
Cases $(B)$ and $(C)$ only occur when $\nu\in W_0\vartheta$.
 If furthermore $j=0$, then also $\alpha_0= -\vartheta$ (so we are necessarily in the untwisted case $\hat R_0=R_0^\vee$).

When $\lambda + \nu\in {P_c}$ the lemma is trivial. Let
 $\lambda +\nu\not\in {P_c}$ and  a decomposition
 $w_{\lambda + \nu}=s_{j_\ell}\cdots s_{j_1}$ with $\ell =\ell (w_{\lambda+\nu})\geq 1$, 
 we define  $\nu_0:=\nu$,   $\nu_k:= s'_{j_{k}}\nu_{k-1}$ (for $k=1,\ldots ,\ell$), $b_0:=a_{j_1}$ and  $b_k=\beta_k^\vee+ r_k {m_\beta}c:= s_{j_1}\cdots s_{j_{k}}a_{j_{k+1}}$ (for $k=1,\ldots ,\ell-1$).
 This means that $R[ {\lambda+\nu}]=\{ b_0,\ldots ,b_{\ell-1}\}$ (cf. \cite[(2.2.9)]{mac:affine}).
By considering the three aforementioned possible cases we have 
\begin{equation*}\label{chain0}
 \lambda + \nu=\lambda + \nu_0\stackrel{s_{j_1}}{\longrightarrow}\lambda + \nu_1\stackrel{s_{j_2}}{\longrightarrow}\cdots  \stackrel{s_{j_{\ell-1}}}{\longrightarrow}
 \lambda + \nu_{\ell-1},
 \end{equation*}
where all these steps involve only $(A)$ or $(B)$, while $(C)$ can only occur at the final step
\begin{flalign}\label{part-i}
&&  \lambda + \nu_{\ell-1} &\stackrel{s_{j_\ell}}{\longrightarrow} \lambda 
 =(\lambda  + \nu)_+ , &&\\
 \noalign{\noindent\text{or does not occur at all}}
\label{part-ii}
 & & \lambda + \nu_{\ell-1} &\stackrel{s_{j_\ell}}{\longrightarrow} \lambda + \nu_\ell =(\lambda + \nu)_+,&&
 \end{flalign}
In situation  \eqref{part-i}  we have that
  $s_{j_\ell}w_{\lambda + \nu}=s_{j_{\ell-1}}\cdots s_{j_1}\in W_{\lambda}$,
  $(\lambda + \nu)_+=\lambda$, and $t_{j_\ell}=t_0$.
 Even more, we have 
 $s^\prime_{j_{\ell}}w'_{\lambda + \nu}\nu=\nu_{\ell -1}=-\alpha_{j_\ell}$, which implies
  $-\nu =(s'_{j_{\ell}}w'_{\lambda + \nu})^{-1}\alpha_{j_\ell}=(s'_{j_1}\cdots s'_{j_{\ell-1}}\alpha_{j_\ell})=\beta_{\ell-1}$, thus
 \begin{multline*}
- \langle\lambda,\nu^\vee\rangle+r_{\ell -1}{m_\nu}c =b_{\ell-1}(\lambda)=((s_{j_{\ell}}w_{\lambda+\nu})^{-1}a_{j_\ell})(\lambda)
 \\=a_{j_\ell}(s_{j_{\ell-1}}\cdots s_{j_1}\lambda)=a_{j_{\ell}}(\lambda)=1,
 \end{multline*}
hence
 \begin{align}\label{nuvee-1}
- \nu^\vee+(1+ \langle \lambda,\nu^\vee\rangle)=&b_{\ell -1}\in R[ {\lambda+\nu}] 
 .
 \end{align}
 On the other hand, in situation  \eqref{part-ii}  we have that  $w_{\lambda + \nu}\in W_{\lambda}$ and $(\lambda + \nu)_+\neq \lambda $.

 In order to compute $\theta (\lambda + \nu )$ let us notice that it has to be the number of times that the case $(B)$ occurs in the steps above, since $\theta ((\lambda + \nu)_+)=0$. In other words,  the number of times that
	  $ \langle \nu_k,\alpha^\vee_{j_{k+1}} \rangle = -2$ for $k=0,\ldots ,\ell^\prime -1$, with 
$\ell^\prime=\ell -1$ in situation    \eqref{part-i}   and $\ell^\prime =\ell$ in situation   \eqref{part-ii}.
 Since for $k=0,\ldots ,\ell^\prime -1$: $$\langle \nu_k,\alpha^\vee_{j_{k+1}} \rangle = -2
 \Leftrightarrow\langle \nu ,\beta^\vee_k \rangle =-2\Leftrightarrow\nu=-\beta_k,$$ and
\begin{multline*}
\langle\lambda,\beta_{k}^\vee\rangle+m_{\beta_k}r_{k} c=b_{k}(\lambda)=(s_{j_1}\cdots s_{j_{k}}a_{j_{k+1}})(\lambda)\\= a_{j_{k+1}}(s_{j_{k}}\cdots s_{j_1}\lambda)=a_{j_{k+1}}(\lambda)=0,
\end{multline*}
then
 $$
 \beta_{k}^\vee-\langle\lambda,\beta_{k}^\vee\rangle =b_{k}\in R[ {\lambda + \nu}].
 $$
Hence, $\theta (\lambda + \nu)$ is equal to $1$ or $0$ when $-\nu^\vee  +\langle\lambda,\nu^\vee\rangle\in R[ {\lambda + \nu}]$ or
 $-\nu^\vee + \langle\lambda,\nu^\vee\rangle\not\in R[ {\lambda + \nu}]$, respectively. 
 In situation   \eqref{part-i}  we have that $\theta(\lambda + \nu)=0$,
 because by Eq. \eqref{nuvee-1} we have $\langle \lambda,\nu^\vee\rangle + 1\in m_\nu c\mathbb{Z}$ and if  $-\nu^\vee+\langle\lambda,\nu^\vee\rangle\in  R[ {\lambda + \nu}] \subset R$ then 
 $\langle \lambda,\nu^\vee\rangle\in m_\nu c\mathbb{Z}$, this would  contradict that $c>1$.

For $\theta(\lambda  + \nu)=1$ we have $-\nu^\vee + \langle \lambda,\nu^\vee\rangle \in R^+$, and therefore $\nu\in R_0\cap P_\vartheta^\star = W_0\vartheta$ and $\langle \lambda,\nu^\vee\rangle \in m_\vartheta c\mathbb{Z}$. 
On the other hand we have that 
$|\langle \lambda,\nu^\vee\rangle |  = |\langle  \lambda, m_\nu \hat\nu\rangle |   \leq m_\vartheta c$ for all $\lambda\in {P_c}$, 
proving 
$
 m_\vartheta \langle \lambda ,\hat \nu\rangle =
 \langle \lambda ,\nu^\vee\rangle\in \{m_\vartheta c,0
 \} $
 if   $\theta(\lambda +\nu)>0$, and this concludes  the proof of the lemma. 


\subsection{Proof of Lemma \ref{Iqm-action:lem} }\label{subsectionLemmaB}
It will be more useful to use the following  reformulation:
  \begin{equation*}\label{Iqm-action:eqn}
 ( T_{w_{\lambda + \nu}} f)((\lambda + \nu)_+)
=  f(\lambda + \nu)   - d_{\lambda ,\nu}(1-t^{-1}_{\vartheta})  f(\lambda) .
 \end{equation*}

For $f\in C(P)$, $\mu\in P$, $j=0,\ldots, n$ with $0\le a_j(\mu)\le 2$ the action of $T_j$ is given explicitly by
 \begin{equation}\label{Ijact}
 (T_jf)(\mu )
  =
  \begin{cases}
  t_j f(\mu )&\text{if}\ a_j(\mu) =0 \\
  f(s_j\mu ) =f(\mu-\alpha_j) &\text{if}\ a_j(\mu) =1 \\
f(\mu-2\alpha_j) -(t_j-1)f(\mu-\alpha_j) &\text{if}\ a_j(\mu) = 2
 \end{cases}.
 \end{equation}
 Now we proceed by induction on $\ell (w_{\lambda + \nu})$. For  $\lambda + \nu\in {P_c}$ the result is trivial. Let assume that $\ell (w_{\lambda+\nu})>1$ 
  and $s_j$ ($0\leq j\leq n$) such that $a_j\in R[{\lambda + \nu}]$, then 
 $\ell (w_{\lambda + \nu}s_j)=\ell (w_{\lambda + \nu})-1$.
 By  the observations made at the beginning of  Section \ref{subsectionLemmaA}we have that 
 $w_{\lambda + \nu}s_j=w_{s_j(\lambda + \nu)}$  with either
 $s_j(\lambda +\nu)=\lambda + s'_j\nu$ (cases $(A)$ and $(B)$) or
 $s_j(\lambda + \nu)=\lambda (\in {P_c})$  (case $(C)$). 
 In the case $(C)$ we have $w_{\lambda + \nu}=s_j$ and the statement to prove is just the case $a_j(\mu)=1$ of  Eq. \eqref{Ijact} with $\mu=\lambda$.
 Furthermore, for the cases $(A)$ and $(B)$ we have
 \begin{align*}
 (T_{w_{\lambda + \nu}} f)((\lambda + \nu)_+)&=
(T_{w_{\lambda + s'_j\nu}} T_jf)((\lambda  + s'_j\nu)_+)  \\  &=
(T_j f) (\lambda + s'_j\nu) -  d_{\lambda,s'_j\nu} (1-t_{\vartheta}^{-1})(T_jf)(\lambda)
 \end{align*}
by the induction hypothesis and the fact $(\lambda  + s'_j\nu)_+=(\lambda + \nu)_+$.

 If we are in the case  $(A)$ then we have $(T_j f) (\lambda+s'_j\nu) =f(\lambda +\nu)$  and $(T_jf)(\lambda )=t_j f(\lambda )$ by the 
 situations $a_j(\mu) =1$ and $a_j(\mu) =0$ of Eq. \eqref{Ijact}, respectively. 
 This finishes the induction step since 
 $d_{\lambda ,s_j^\prime\nu}=d_{\lambda ,\nu}t_j^{-1} $. For $\nu\in P_\vartheta^\star\backslash W_0\vartheta$ this follows from $d_{\lambda ,s_j^\prime\nu}=d_{\lambda ,\nu}=0$,  while for $\nu\in W_0\vartheta$ it follows from   $\theta (\lambda+s'_j\nu)=\theta (\lambda+\nu)$ and that for $j>0$ we have $e_t(s_j\nu)=e_t(\nu)t_j$
 and $\langle \lambda,s_j  \hat\nu \rangle=
 \langle s_j\lambda,\hat\nu \rangle=\langle \lambda, \hat\nu \rangle$, while on the other hand, for $j=0$ we have
 $e_t(s'_0\nu)=e_t(\nu+\alpha_0)=e_t(\nu) h_t^{-1} t_{\vartheta}$ 
 and  (if $\theta(\lambda+\nu)>0$)  also $\langle \lambda,s'_0 \hat\nu \rangle=
 \langle s'_0\lambda, \hat\nu \rangle=\langle \lambda+c\alpha_0,\hat\nu \rangle =
 \langle \lambda, \hat\nu \rangle+c\langle \alpha_0,\hat\nu \rangle=\langle \lambda, \hat\nu \rangle -  c$
 and therefore
 $\text{sign} (\langle \lambda,s_0^\prime \hat \nu\rangle)=\text{sign} (\langle \lambda, \hat\nu\rangle)-1$ (cf. Lemma. \ref{theta:lem}).
 
  If we are in the case $(B)$ then we have that $\nu\in W_0\vartheta$, $t_j=t_{\vartheta}$ and 
  $(T_j f) (\lambda+s'_j\nu) =f(\lambda+\nu)-t_\vartheta(1-t_\vartheta^{-1})f(\lambda )$ because of the case $a_j(\mu) =2$ of Eq. \eqref{Ijact} for $\mu=\lambda+s'_j\nu$. We also have 
 $d_{\lambda,s'_j\nu}=0$ because $0\leq\theta (\lambda+s'_j\nu )<\theta (\lambda+\nu)\leq 1$. To complete the induction it remains to prove 
 that  $d_{\lambda ,\nu}=t_\vartheta$. 
For this observe that  $\theta (\lambda+\nu)=1$ and when $j>0$ we have
 $e_t(-\nu)=e_t(\alpha_j)=t_j=t_\vartheta$ (as $e_t (\alpha_j)=  e_t(s_j\alpha_j)t_j^{\langle \alpha_j,\alpha_j^\vee\rangle} 
 = e_t (-\alpha_j)t_j^2=t_j^2/e_t (\alpha_j)$) and $\langle \lambda ,\hat \nu\rangle=-\langle \lambda,\hat\alpha_j\rangle=0$, while for  $j=0$ we have $e_t (\nu)=e_t (-\alpha_0)=e_t (\vartheta)=h_t/t_\vartheta$ (since in this case $\alpha_0=-\vartheta$, cf. Lemma \ref{theta:lem}) and $\langle \lambda,\hat\nu\rangle=-\langle \lambda,\hat \alpha_0\rangle=-\langle \lambda, \alpha^\vee_0\rangle= -a_0(\lambda)+c=c>0$.

 \section{The structure constants revisited}\label{sec9}
The computation in the previous section produces the coefficients of the Pieri rule from the action of 
$L_\omega$ \eqref{Lp} in $\mathcal{C}(P_c)$. In principle the same strategy  can be followed to compute the structure constants $c^{\nu, (c)}_{\lambda ,\mu}(t)$ ($\lambda,\mu,\nu\in P_c$) more generally. To this end one starts with the monomial expansion of the Macdonald spherical function $M_\lambda (\xi)$, $\lambda\in P^+$:
\begin{equation}\label{m:exp}
M_\lambda (\xi)=\sum_{\mu\in P^+,\, \mu \leq \lambda}   n_{\lambda ,\mu } (t) m_\mu (e^{i \xi}) ,
\end{equation}
 where we have employed the dominance partial order on $P^+$: $\mu\leq\lambda$ iff $\lambda-\mu\in Q^+$.
With the aid of the expansion coefficients $n_{\lambda ,\mu}(t)$ one defines the following operator-valued Macdonald spherical function $M_\lambda (L):\mathcal{C}(P)\to\mathcal{C}(P)$ via the formula:
\begin{equation}\label{OM}
M_\lambda (L)=\sum_{\mu\in P^+,\, \mu \leq \lambda}   n_{\lambda ,\mu } (t) L_\mu 
\end{equation}
(cf. Eqs. \eqref{free-laplacian}, \eqref{Lp}).  For $\lambda\in P_c$, the operator-valued Macdonald spherical function  $M_\lambda (L)$ \eqref{OM} acts as a linear difference operator in the invariant subspace
$\mathcal{C}(P)^W\cong \mathcal{C}({P_c})$ with coefficients  given by the structure constants $\text{c}^{\nu, (c)}_{\lambda ,\mu}(t)$ ($\mu,\nu\in P_c$).

\begin{theorem}[Structure Constants]
For any $\lambda\in P_c$, the action of $M_\lambda(L)$ on  $f\in \mathcal{C}(P_c)$  is given by
\begin{equation}\label{sc}
(M_\lambda(L) f) (\mu) =  \sum_{\nu\in P_c} \text{c}^{\nu, (c)}_{\lambda ,\mu}(t) f(\nu) ,
\end{equation}
with  $ \text{c}^{\nu, (c)}_{\lambda ,\mu}(t)$ as defined in Eq. \eqref{LRcoef}.
\end{theorem}
\begin{proof}
By linearity, it suffices to verify Eq. \eqref{sc} on the basis of Macdonald spherical functions $\Phi_\xi$, $\xi\in\mathcal{P}_c$. To this end we compute for $\xi\in\mathcal{P}_c$:
\begin{align*}
M_\lambda(L)\Phi_\xi  & \stackrel{\text{Eq}. \eqref{OM}}{=}
\sum_{\mu\in P^+,\, \mu \leq \lambda}   n_{\lambda ,\mu } (t) L_\mu \Phi_\xi  \stackrel{\text{Prop}. \ref{diagonal:thm}}{=}
\sum_{\mu\in P^+,\, \mu \leq \lambda}   n_{\lambda ,\mu } (t) m_\mu (e^{i \xi})  \Phi_\xi  \\
& \stackrel{\text{Eq}. \eqref{m:exp}}{=}
M_\lambda(\xi)\Phi_\xi  =  M^{(c)}_\lambda(\xi) \Phi_\xi  .
\end{align*}
Evaluation of this identity at $\mu\in P_c$  with the aid of Eq. \eqref{Msf-evaluation}
entails the desired formula for $f=\Phi_\xi$:
\begin{align*}
(M_\lambda(L)\Phi_\xi )(\mu)& = M^{(c)}_\lambda(\xi) \Phi_\xi (\mu)=
M^{(c)}_\lambda(\xi) M^{(c)}_\mu(\xi)=
\sum_{\nu\in P_c} \text{c}^{\nu, (c)}_{\lambda ,\mu}  M^{(c)}_\nu (\xi)\\
&=
\sum_{\nu\in P_c} \text{c}^{\nu, (c)}_{\lambda ,\mu}   \Phi_\xi (\nu) .
\end{align*}
\end{proof}

\vspace{3ex}
\noindent {\bf Acknowledgments.} We thank the referee for the helpful constructive remarks and suggesting some improvements concerning the presentation.

\bibliographystyle{amsplain}

\end{document}